%% file: categorified_spectra.tex
\documentclass[a4paper,10pt, oneside]{amsart}
\input{preamble}

\input{macros}
\title{}

\begin{document}

\title{Every spectrum is the K-theory of a stable $\infty$-category}
\author{Maxime Ramzi}
\address{K\o benhavns Universitet, Institut for matematiske fag \\
Universitetsparken 5\\
2100 K\o benhavn \O \\
Denmark
}
\email{\href{mailto:ramzi@math.ku.dk}{ramzi@math.ku.dk}}
\author{Vladimir Sosnilo}
\address{
%M309\\
Fakult\"at f\"ur Mathematik \\
Universit{\"a}t Regensburg\\
%Universit{\"a}tsstra{\ss}e 31 \\
93040 Regensburg\\
Germany
}
\email{\href{mailto:vsosnilo@gmail.com}{vsosnilo@gmail.com}}
\author{Christoph Winges}
\address{
%M115\\
Fakult\"at f\"ur Mathematik \\
Universit{\"a}t Regensburg\\
%Universit{\"a}tsstra{\ss}e 31 \\
93040 Regensburg\\
Germany
}
\email{\href{mailto:christoph.winges@ur.de}{christoph.winges@ur.de}}

\bibliographystyle{alphamod}

\begin{abstract}
We prove that any spectrum is equivalent to the nonconnective K-theory of a stable $\infty$-category. 
We use these results to construct a stable $\infty$-category $\sC$ with a bounded t-structure such that 
$\K(\sC)$ is not equivalent to $\K(\sC^\heartsuit)$, disproving a conjecture of Antieau, Gepner, and Heller.
\end{abstract}

\maketitle
{}

\setcounter{tocdepth}{1}
\tableofcontents

\section*{Introduction}

\ssec{} 
Throughout the history of mathematics it has been observed that important numerical %numeric??
invariants of mathematical objects admit {\it categorifications}. For example, the Euler characteristics and the Betti 
numbers of a space are categorified by the homology, the Hilbert polynomial of a scheme is categorified by the coherent 
cohomology, and various knot-theoretic invariants are categorified by Khovanov homology \cite{KhovanovJones} and Knot-Floer 
homology \cite{OszvathSzabo}. 
More precisely, categorifying an invariant valued in some abelian group $A$ means finding a stable $\infty$-category 
$\sC_A$ with $\K_0(\sC_A) \simeq A$ such that the given invariant lifts to %some kind of
functor valued in $\sC_A$. In this situation, $\sC_A$ is thought of as a categorification of $A$. 
In \cite{Khovanov_categotification} Mikhail Khovanov asked several questions about the extent to which categorification is 
possible. In particular, he asked whether it is possible to categorify the rational numbers, that is, to find a stable 
$\infty$-category $\sC_\mbf{Q}$ with $\K_0(\sC_\mbf{Q}) \simeq \mbf{Q}$. On the level of abelian groups, this question was eventually answered in the paper 
\cite{Barwick2019}. More generally, for a given stable $\infty$-category $\sC$ and a set of primes $S \subset \Z$ they 
construct a stable $\infty$-category $S^{-1}\sC$ satisfying 
\[
\K(S^{-1}\sC) \simeq S^{-1}\K(\sC).
\]
In this paper we show that every spectrum admits a categorification. 

\begin{thmX}[Theorem~\ref{thm:cat_spectrum}]\label{thmX:cat_spectrum}
For every spectrum $M$ there exists a small idempotent complete stable $\infty$-category $\sC_M$ such that 
\[
\K(\sC_M) \simeq M.
\]
Here $\K$ denotes the nonconnective K-theory spectrum. Moreover, the construction of $\sC_M$ is functorial in $M$.
\end{thmX}

One could also view this result as a variant of Thomason's theorem that says that every connective spectrum can be 
recovered as the group completion of the underlying monoid of a symmetric monoidal 1-category 
\cite{Thomason1995}. %[Theorem~5.1]

To prove this theorem, we first show that several natural constructions one may perform 
with the K-theory spectrum can be naturally lifted to constructions on the categorical level. These include:
\begin{enumerate}
\item suspensions (Proposition~\ref{prop:categoifying_cofiber});
\item loops (Theorem~\ref{thm:categorifying_loops});
\item cofibers of maps induced by functors (Proposition~\ref{prop:categoifying_cofiber}).
\end{enumerate}
The rest of the proof uses the Dundas--McCarthy formula for computing 
$\mrm{THH}(\sphere; M) \simeq M$ (Theorem~\ref{thm:DM}) and carefully performs all the constructions involved in it at the 
categorical level. 
%We also note that our construction of $\sC_M$ ends up being functorial in $M$. 

\ssec{} 
One of the most fundamental results about algebraic K-theory is the {\it theorem of the heart}. 
In Barwick's formulation \cite{Barwick_heart}, it says that for any small stable $\infty$-category $\sC$ with a bounded t-structure, there is a canonical equivalence 
\[
\K^{\mrm{cn}}(\sD^b(\sC^\heartsuit)) \simeq \K^{\mrm{cn}}(\sC). 
\]
Here $\K^{\mrm{cn}}$ stands for connective K-theory. Antieau, Gepner, and Heller observed that this equivalence 
can be extended to nonconnective K-theory, as long as the heart $\sC^\heartsuit$ is a noetherian abelian 
category \cite[Theorem~1.3]{AGH}. %They did so by proving that the negative K-groups in this case vanish. 
Their Conjecture~C states that the {\it nonconnective theorem of the heart} holds in general, i.e. the map 
\[
u:\K(\sD^b(\sC^\heartsuit)) \to \K(\sC)
\]
is always an equivalence, without any assumptions on the heart. This was also asked by 
Burklund and Levy \cite[Question~1.4]{Burklund_2023}. As the main application of our results, we show that it is not 
true.
\begin{thmX}[\ref{sssec:counterexample}]\label{thmX:tohisfalse}
The nonconnective theorem of the heart is false: there exists a small stable $\infty$-category with a bounded t-structure such that the map $u$ is not an equivalence. 
\end{thmX}

The construction of a counterexample is
surprisingly easy, once we have Theorem~\ref{thmX:cat_spectrum}. We start 
by building a stable $\infty$-category $\sC$ whose K-theory has sufficiently nontrivial {\it chromatic} behaviour. 
Given such an $\infty$-category $\sC$ we consider a new small stable $\infty$-category $\sA c(\sC)$ that admits a 
bounded t-structure\footnote{This construction is an adaptation of the construction of Neeman 
\cite[Section~1]{Neeman2021} to the more 
general setting.} %(see also \cite[Proposition~10.1]{Schlichting2006}).}  
(see Section~\ref{sec:toh_is_false}), and whose K-theory spectrum happens to be chromatically similar to $\K(\sC)$. 
This immediately yields a counterexample, as the chromatic behaviour of $\K(\sD^b(\sC^\heartsuit))$ is always very 
simple. 

Using the same example we also show that the nonconnective K-theory of a stable 
$\infty$-category with a bounded t-structure is not $\mathbb{A}^n$-invariant (Remark~\ref{rem:K_notA1invariant}). 

\ssec{Notation}

\begin{itemize}
\item We will denote by $\sS$ the $\infty$-category of spaces and by $\Spt$ -- the $\infty$-category of spectra. 
\item We will denote by $\Cat^{\mrm{st}}_{\infty}$ the $\infty$-category of small stable $\infty$-categories and 
by $\Cat^{\perf}_{\infty}$ the subcategory of idempotent complete $\infty$-categories. 
For $\sC, \sD \in \Cat^{\mrm{st}}_{\infty}$ we denote by $\Funex(\sC,\sD)$ the 
$\infty$-category of (left) exact functors from $\sC$ to $\sD$. The $\infty$-category of presentable stable 
$\infty$-categories is denoted by $\PrLst$.
\item For $\infty$-categories with products 
$\sC, \sD$ we denote by $\Fun_\times(\sC,\sD)$ the $\infty$-category of product-preserving 
functors from $\sC$ to $\sD$. We will denote by $\Cat^{\mrm{padd}}_{\infty}$ the $\infty$-category of small additive $\infty$-categories which are idempotent complete. 
%\item Given a small additive $\infty$-category $\sA$ we denote by $\hat{\sA}$ the $\infty$-category of 
%additive presheaves of spectra on $\sA$. The smallest full subcategory of $\hat{\sA}$ containing the Yoneda 
%image of $\sA$ and closed under finite limits and finite colimits is denoted by $\sA^\mrm{fin}$.
\item Given a stable $\infty$-category $\sC$ and $X,Y \in \sC$ we denote by $\map_\sC(X,Y)$ the  
corresponding mapping spectrum. 
\item We use the notation $\K(-)$ for the {\bf nonconnective} K-theory spectrum. The connective K-theory spectrum, as 
defined in \cite{BarwickKtheory}, is denoted with $\K^{\mrm{cn}}(-)$. Note that $\K^{\mrm{cn}}(-)$ is not invariant under idempotent completion.
\end{itemize}

\ssec{Acknowledgements}

We want to thank Benjamin Antieau, Denis-Charles Cisinski, Sasha Efimov, Niklas Kipp, Ishan Levy and Marco Varisco for useful 
conversations and their comments on an early draft of the paper. 
The first author is supported by the Danish National Research Foundation through the Copenhagen Centre for Geometry and Topology (DNRF151).
The second and third author are supported by the SFB 1085 ``Higher Invariants'' funded by the Deutsche Forschungsgesellschaft (DFG). 

\section{Reminder on localizing invariants}\label{sec:noncommutative_motives}

In our exposition we mostly follow the terminology of \cite{blumberg2013universal} and \cite{HSS}. Note, however, 
that we consider a 
slightly more general notion of a localizing invariant (as in \cite{Land_2019} or \cite{BKRS}).

\begin{defn}\label{defn:KV_sequence}
We say that a diagram in $\Cat^{\mrm{st}}_{\infty}$
\[
\sA \stackrel{f}\to \sB \stackrel{g}\to \sC
\]
is a {\bf Karoubi--Verdier sequence} %of idempotent complete small stable $\infty$-categories 
if 
$f$ is fully faithful, the composite $g\circ f$ is trivial and the induced functor 
\[
\sB/\sA \to \sC
\]
is an equivalence up to idempotent completion. 

We say that it is a {\bf Verdier sequence} if $g$ is essentially 
surjective and the essential image of $f$ is closed under retracts. 
\end{defn}

\begin{defn}\label{defn:localizing}
Let $\sE$ be a stable $\infty$-category. A functor 
\[
E: \Cat^{\perf}_{\infty} \to \sE
\]
is said to be a {\bf localizing invariant} if for any Karoubi-Verdier sequence $\sA \to \sB \to \sC$ in 
$\Cat^{\perf}_{\infty}$ the sequence
\[
E(\sA) \to E(\sB) \to E(\sC)
\]
is a fiber sequence.
\end{defn}

\begin{defn}\label{defn:finitary}
Let $\sE$ be a cocomplete stable $\infty$-category. We say that a functor 
\[
E: \Cat^{\perf}_{\infty} \to \sE
\]
is {\bf finitary} if it preserves filtered colimits. For a cocomplete stable $\infty$-category $\sD$ we denote the subcategory of $\Fun(\Cat^{\perf}_{\infty}, \sD)$ consisting of finitary localizing invariants 
by $\Fun^{\mrm{loc}, \mrm{fin}}(\mrm{Cat}^{\perf}_{\infty}, \sD)$.
\end{defn}

\begin{exam}
The nonconnective K-theory functor $\K : \Cat^\perf_\infty \to \Spt$ is a finitary localizing invariant. 
\end{exam}

\begin{thm}[{Blumberg, Gepner \& Tabuada \cite[Theorem~8.7]{blumberg2013universal}}]\label{thm:universal_property_motives}
There exists a functor 
\[
\mathcal{U}_{\mrm{loc}}: \Cat^\perf_\infty \to \Motnc
\]
initial among finitary localizing invariants with values in a presentable stable $\infty$-category. 
\end{thm}

The $\infty$-category $\Motnc$ is naturally presentably symmetric monoidal and the functor 
$\mathcal{U}_{\mrm{loc}}$ admits a symmetric monoidal enhancement upon considering the symmetric monoidal 
structure of \cite[4.1.2]{BFN} on $\Cat^\perf_\infty$ (see also \cite[Section~3.1]{blumberg2013universal}). 
This structure allows us to identify the nonconnective K-theory naturally in this setting:

\begin{thm}[{Blumberg, Gepner \& Tabuada \cite[Theorem~9.8]{blumberg2013universal}}]\label{thm:universality_Ktheory}
The unit $\unitm \in \Motnc$ is compact and for any stable $\infty$-category $\sC$ there is a natural 
equivalence 
\[
\map_{\Motnc}(\unitm, \mathcal{U}_{\mrm{loc}}(\sC)) \simeq \K(\sC).
\]
\end{thm}

The following easy corollary of Theorem~\ref{thm:universality_Ktheory} will come in handy later in the paper.

\begin{cor}\label{cor:autoequivalencesKth}
Any natural equivalence $\K \to \K$ is either homotopic to $\id$ or $-\id$. 
\end{cor}
\begin{proof}
By Theorems~\ref{thm:universal_property_motives}~and~\ref{thm:universality_Ktheory} we have an equivalence of rings 
\[
\pi_0\mrm{Nat}(\K,\K) \simeq \pi_0\K(\sphere) \simeq \Z.
\]
The result follows since the only units on the right-hand side are 1 and -1.
\end{proof}

\section{\texorpdfstring{Reminder on the Dundas--McCarthy theorem}{Reminder on the Dundas-McCarthy theorem}}
In this section, we remind the reader of a special case of a nonconnective variant of the Dundas--McCarthy theorem. %need to add a ref, in fact refs. Also decide what we do here in detail: I can see a few options : 1- Only what we need, this involves a 2 line proof + reference to literature; 2- Only what we need, but also state the general version, refer to Raskin for a sketch, and to the upcoming work of Harpaz-Nikolaus-Saunier/my upcoming work for the general version; 3- Actually prove the general version (probably a bad idea because the set up might take a while). The current write-up goes for a variant of 2-, but can easily be modified to just be 2-. 
To state this result, we need to introduce twisted endomorphism $\infty$-categories. 
\begin{defn}\label{defn:bimodule}
Let $\sC$ be a small stable $\infty$-category. By a $\sC$-{\bf bimodule} we simply mean an exact functor $\sC \to \Ind(\sC).$
\end{defn}

\begin{rem}
A $\sC$-bimodule $T$ gives rise to a colimit-preserving functor 
\[
T: \Ind(\sC) \to \Ind(\sC). 
\]
We often abuse the notation and denote the latter by the same letter.
\end{rem}

\begin{exam}\label{ex:tensor}
If $M$ is an $R$-bimodule for some ring spectrum $R$, $M \otimes_R - : \mbf{Perf}_R \to\Ind(\mbf{Perf}_R)$ defines 
a $\mbf{Perf}_R$-bimodule. 
\end{exam}

\begin{defn}[{\cite[Definition~II.1.4]{nikolaus-scholze}}]\label{defn:twistEnd}
Let $\sC$ be a small stable $\infty$-category, and $T: \sC\to \Ind(\sC)$ a $\sC$-bimodule. 
The $\infty$-category $\End(\sC;T)$ of {\bf twisted endomorphisms} is the lax equalizer of the 
Yoneda embedding and $T$:
\[
\begin{tikzcd}
\sC \arrow[r, "T", swap, shift right]\arrow[r, "\yo", shift left] & \Ind(\sC).
\end{tikzcd}
\]
%$\yo: C\to \Ind(C)$ and $T:C\to \Ind(C)$ 
In other words, $\End(\sC;T)$ is the pullback of the cospan 
\[
\begin{tikzcd}
& \Fun(\Delta^1,\Ind(\sC))\arrow[d, "{(s,t)}"]\\
\sC \arrow[r, "{(\yo,T)}"] & \Ind(\sC) \times \Ind(\sC).
\end{tikzcd}
\]
Its objects are pairs $(x, f: x\to Tx), x\in \sC$, and its morphisms are given by commutative squares. 
\end{defn}

\begin{rem}\label{rem:endsqz}
When $\sC= \mbf{Perf}_A$ and $T = \Sigma M\otimes_A-$ for $A$ a ring spectrum and $M$ an $A$-bimodule, there is a 
fully faithful embedding $\mbf{Perf}_{A\oplus M} \to \End(\mbf{Perf}_A; \Sigma M\otimes_A -)$ whose essential 
image consists of nilpotent twisted endomorphisms, i.e. objects $(P, P \to \Sigma M\otimes_A P)$
such that for $n>>0$, the composite 
\[
P\to \Sigma M\otimes_A P\to \Sigma^2 M^{\otimes_A 2}\otimes_A P\to\dots \to \Sigma^n M^{\otimes_A n}\otimes_A P
\] is null. In particular, if $A$ and $M$ are connective, this fully faithful embedding is an equivalence. We refer to 
\cite[Proposition 3.2.2]{raskin2018dundasgoodwilliemccarthy} 
for a sketch of proof, and to 
\cite[Theorem 2.34]{barkan2022explicit} for a more detailed proof. 
\end{rem}
We will use the above remark specifically when $A= \sphere$ and $M=\Sigma^n\sphere$ for some $n$.

\begin{constr}
Let $E$ be a functor from $\Cat^\perf_\infty$ to a stable $\infty$-category, and let $T$ be a $\sC$-bimodule for some small 
(idempotent complete) stable $\infty$-category $\sC$. There is a 
retraction of $\End(\sC;T)$ onto $\sC$: the inclusion is given by sending $x$ to $(x, x\xrightarrow{0} Tx)$, and the 
retraction by sending $(x,x\to Tx)$ to $x$. 
We let $\widetilde E(C;T)$ denote the cofiber of the map $E(\sC) \to E(\End(\sC;T))$ induced by the inclusion. It is a direct summand of $E(\End(\sC;T))$, so it 
can be equivalently expressed as the fiber of the retraction. 

For $\sC = \Spt^\omega = \mbf{Perf}_{\mbf{S}}$ and $T = M \otimes -$, we will abbreviate $\widetilde{E}(\End(\Spt^\omega;M)) := \widetilde{E}(\End(\Spt^\omega;M \otimes -))$.
\end{constr}

We recall the following from \cite[Theorem~6.1.1.10, Example~6.1.1.28]{HA}:
%\cite[Example 6.1.1.28, Lemma 6.1.1.33, Proof of Theorem 6.1.1.10]{HA}:
\begin{lem}\label{lem:P1}
Let $F: \sC\to \sD$ be a pointed functor between stable $\infty$-categories, where $\sD$ admits sequential colimits. Its first 
Goodwillie derivative, $P_1F$, defined as the initial exact functor below $F$, is given by the formula
\[
\colim_n \Omega^n F(\Sigma^n -).
\]
In particular, the latter is exact. 
\end{lem}
In this language, the main result of this section is:
\begin{thm}\label{thm:DM}
The functor $P_1\widetilde \K(\Spt^\omega; -) : \Spt\simeq\Fun^{\LL}(\Spt,\Spt)\to \Spt$ is equivalent to the identity, i.e. we have a natural equivalence $$M\simeq \colim_n \Omega^n\widetilde \K(\Spt^\omega;\Sigma^n M)$$
\end{thm}
\begin{rem}
This theorem is a special case of a more general theorem comparing $P_1\widetilde \K(\sC;-)$ and $\mathrm{THH}(\sC;-)$ 
for a fixed $\sC$; in fact, this comparison can be made natural in $\sC$, too. We make a few remarks about this 
more general statement: using trace theories following Nikolaus' talk in \cite{OWRTC}, one can give a completely 
natural and structured proof of this equivalence. Such a proof will appear in a forthcoming work of the first named 
author; in fact, using much less technology (but similar ideas), and for a fixed $\sC$, one can \emph{deduce} the 
general result from the special case $\sC=\Spt^\omega$. A different proof will also appear in a forthcoming work of 
Harpaz, Nikolaus and Saunier. Finally, we point out that a proof is sketched in 
\cite[Theorem 3.10.1]{raskin2018dundasgoodwilliemccarthy}. 

We also note that our proof here will in fact \emph{use} the original theorem of Dundas--McCarthy \cite{Dundas1994}, 
but the aforementioned approaches do not, and deduce it as a special case. 
\end{rem}
\begin{rem}
In the setting of Remark~\ref{rem:endsqz}, and for a simplicial ring $A$, the result is simply the Dundas--McCarthy theorem, proved in \cite{Dundas1994}.
Note that Dundas and McCarthy obtain $\mathrm{THH}(A,\Sigma M)$ rather than $\mathrm{THH}(A,M)$ precisely because of the $\Sigma$ appearing in Remark \ref{rem:endsqz}. 
\end{rem} 

\begin{lem}\label{lem:endfilt}
Let $\sC$ be a small stable $\infty$-category. The functor 
\[
\Fun_{\mrm{ex}}(\sC,\Ind(\sC))\to \Cat^\perf_\infty
\] 
\[
T\mapsto \End(\sC;T)
\]
preserves filtered colimits. 
\end{lem}
\begin{proof}
Let $T_\bullet: I\to \Fun_{\mrm{ex}}(\sC,\Ind(\sC))$ be a filtered diagram, let $T:=\colim_I T_i$ and consider the assembly map 
$\colim_I \End(\sC;T_i)\to \End(\sC;T)$. We first note that a simple calculation on mapping spaces proves that this is fully 
faithful. Now given $(x,x\to Tx)$ in the target, as $x\in \sC$ is compact in $\Ind(\sC)$, we find that the map 
$x\to Tx$ factors through $T_i x$ for some $i$, so that $(x,x\to Tx)$ is in the image of $\End(\sC;T_i)\to\End(\sC;T)$. 
\end{proof} 

\begin{proof}[Proof of Theorem \ref{thm:DM}]
By Lemma \ref{lem:P1} (or by definition), the functor $P_1\widetilde \K(\Spt^\omega,-):\Spt\to \Spt$ is an exact functor, and by Lemma \ref{lem:endfilt} together with the fact that $\K$-theory preserves filtered colimits, it is a filtered colimit-preserving functor. All in all, we are trying to identify a colimit-preserving functor $\Spt\to \Spt$. Such a functor is completely determined by the image of $\sphere$, in our situation, by $\colim_n\Omega^n \widetilde \K(\Spt^\omega;\Sigma^n\sphere)$. 

By Remark \ref{rem:endsqz}, this is equivalently given by $\colim_n\Omega^n \widetilde \K(\sphere\oplus\Sigma^{n-1}\sphere)$, and we simply claim this is the sphere spectrum. This is now essentially the main result of \cite{Dundas1994}. Strictly speaking, the authors in \cite{Dundas1994} only prove this for simplicial rings and simplicial 
bimodules, but the case of ring spectra is proved in \cite[Theorem 5.3.5.1]{dgm-local}, where they deduce it from that case by some 
connectivity arguments. 
Alternatively, one can use the methods of \cite{WaldK}, or one can compare $\sphere \oplus\Sigma^{n-1}\sphere$ and $\sphere[\Omega(\Sigma^n S^0)]$ (there is a highly connected map between the two) and reduce to the original results of Goodwillie \cite{GoodwillieK}.

In any case, it follows that in this case, $P_1\widetilde\K(\Spt^\omega; -) \simeq \id_{\Spt}$. 
\end{proof}

\section{Categorification of spectra}\label{}

The goal of this section is proving the following theorem.

\begin{thm}\label{thm:cat_spectrum}
Let $M$ be a spectrum. Then there exists a small idempotent complete stable $\infty$-category 
$\sC_M$ such that 
\[
\K(\sC_M) \simeq M.
\]
Moreover, the assignment $M \mapsto \sC_M$ is functorial.
\end{thm}

\ssec{Reminder on the Calkin construction}

For a small stable $\infty$-category $\sC$ we denote by 
$\mrm{Calk}(\sC)$ the $\omega_1$-small Calkin category of $\sC$, i.e. the idempotent completion of the Verdier quotient of the Yoneda embedding
\[
\sC \to \Ind(\sC)^{\omega_1}.
\]
Here $(-)^{\omega_1}$ denotes the subcategory of $\omega_1$-compact objects. Note that 
$\Ind(\sC)^{\omega_1}$ and $\mrm{Calk}(\sC)$ are small $\infty$-categories.

\ssec{Categorifying suspensions and loops}

Observe that $\Ind(\sC)^{\omega_1}$ %satisfies the
admits an Eilenberg swindle and, in particular, 
$\K(\Ind(\sC)^{\omega_1}) \simeq 0$.
This implies that $\mrm{Calk}$ categorifies the suspension:

\begin{prop}\label{prop:categorifying_suspensions}
Let $\sC$ be a small idempotent complete stable $\infty$-category. Then 
\[
\K(\mrm{Calk}(\sC)) \simeq \Sigma \K(\sC).
\]
\end{prop}

Less trivially, it is also possible to categorify loops.

\begin{thm}\label{thm:categorifying_loops}
There exists a functor $\Gamma : \Cat^\perf_\infty \to \Cat^\perf_\infty$ 
and a canonical equivalence 
\[
\K(\Gamma \sC) \simeq \Omega \K(\sC).
\]
\end{thm}

\sssec{} To prove Theorem~\ref{thm:categorifying_loops} we need to recall a few constructions 
from \cite{KasprowskiWinges}. 
Denote by $F\sC \subseteq \Fun(\mbf{N},\sC)$ the full subcategory of those filtered objects in $\sC$ which stabilize after finitely many steps, and let $F^q\sC \subseteq F\sC$ be the full subcategory of those filtered objects which stabilize at $0$.
 Sending an essentially finite filtration to its associated graded defines a functor $\mrm{gr} \colon F\sC \to \bigoplus_{\mbf{N}} \sC$.
 Define the $\infty$-categories $B\sC$, $B^t \sC$ and $B^q \sC$ of \emph{binary complexes}, \emph{binary top-acyclic  complexes} and \emph{binary acyclic complexes} by the following pullbacks:
 \[\begin{tikzcd}
  B\sC\ar[r, "\top"]\ar[d, "\bot"'] & F\sC\ar[d, "\mrm{gr}"] & & B^t\sC\ar[r, "\top"]\ar[d, "\bot"'] & F^q\sC\ar[d, "\mrm{gr}"] & & B^q\sC\ar[r, "\top"]\ar[d, "\bot"'] & F^q\sC\ar[d, "\mrm{gr}"] \\
  F\sC\ar[r, "\mrm{gr}"] & \bigoplus_{\mbf{N}} \sC & & F\sC\ar[r, "\mrm{gr}"] & \bigoplus_{\mbf{N}} \sC & & F^q\sC\ar[r, "\mrm{gr}"] & \bigoplus_{\mbf{N}} \sC
 \end{tikzcd}\]
 Note that we have natural fully faithful functors $B^q\sC \to B^t\sC \to B\sC$.
 The identity functors induce natural transformations $\Delta \colon F \Rightarrow B$ and $\Delta \colon F^q \Rightarrow B^q$.

Kasprowski and the third author have obtained the following result.

\begin{thm}[{\cite[Theorem~1.1]{KasprowskiWinges}}]\label{thm:almost_categorifying_loops}
We have a canonical equivalence 
\[
\Cofib(\K(F^q\sC) \stackrel{\K(\Delta)}\to \K(B^q\sC)) \simeq \Omega \K(\sC).
\]
\end{thm}

Now to deduce Theorem~\ref{thm:categorifying_loops} it suffices to show that one can categorify the 
cofiber of the map $\K(\Delta)$ in Theorem~\ref{thm:almost_categorifying_loops}. It turns 
out that even the following more general claim holds.

\begin{prop}\label{prop:categoifying_cofiber}
Let $F:\sC \to \sD$ be an exact functor of small %idempotent complete
stable $\infty$-categories.
Then there exist a fully faithful exact functor $G:\sC \to \sD'$ and an exact functor $P:\sD' \to \sD$ such that $P \circ G \simeq F$ and $P$ fits into a right-split Verdier sequence $\Ind(\sC)^{\omega_1} \to \sD' \stackrel{P}{\to} \sD$.
This factorization can be chosen functorially in $F$.

In particular, there exist a small %idempotent complete
stable $\infty$-category $\mrm{Cone}(F)$ and 
an exact functor $\sD \to \mrm{Cone}(F)$ that induces an equivalence 
\[
\Cofib(\K(\sC) \stackrel{\K(F)}\to \K(\sD)) \simeq 
\K(\mrm{Cone}(F)).
\]
\end{prop}
\begin{proof}
Recall the following construction by Georg Tamme \cite{tamme-excision}.
Consider a commutative square of stable $\infty$-categories
\[
\begin{tikzcd}
\sA \arrow[r,"f"]\arrow[d, "p"'] & \sB \arrow[d,"g"] \\
\sA' \arrow[r, "q"] & \sB'
\end{tikzcd}
\]
such that the map $\sA \to \sA' \times_{\sB'} \sB$ is fully faithful. 
Denote by $\sA' \stackrel{\to}\times_{\sB'} \sB$ the lax pullback, i.e. the pullback of the diagram
\[
\begin{tikzcd}
& \Fun(\Delta^1, \sB') \arrow[d, "{(s,t)}"]\\
\sA' \times \sB \arrow[r] & \sB' \times \sB'.
\end{tikzcd}
\]
There are projection functors from $\sA' \stackrel{\to}\times_{\sB'} \sB$ to $\sA$ and $\sB$ split by 
fully faithful inclusions in the opposite directions.
Under our assumptions, we also have a fully faithful functor $\sA \to \sA' \stackrel{\to}\times_{\sB'} \sB$ sending 
$a$ to $(p(a), f(a), qp(a) \simeq gf(a))$ 
and a Karoubi--Verdier sequence
\[
\sA \to \sA' \stackrel{\to}\times_{\sB'} \sB \to \sA' \odot^\sA_{\sB'} \sB.
\]
Moreover, the composite functor $\sA \to \sA' \stackrel{\to}\times_{\sB'} \sB \to \sB$ is equivalent to $f$.
It follows from \cite[Proposition~10]{tamme-excision} that we have an equivalence 
$\K(\sA' \stackrel{\to}\times_{\sB'} \sB) \to \K(\sA') \oplus \K(\sB)$ induced by the two projections. Using this and the fact that $\K(-)$ sends Karoubi--Verdier 
sequences to cofiber sequences, we obtain a canonical equivalence 
\[
\K( \sA' \odot^\sA_{\sB'} \sB)
\simeq 
\Cofib(\K(\sA) \to \K(\sA') \oplus \K(\sB)).
\]
Now for any exact functor $F:\sC \to \sD$ we can apply the above construction to the square 
\[
\begin{tikzcd}
\sC\arrow[d, "F"']\arrow[r] & \Ind(\sC)^{\omega_1}\arrow[d]\\
\sD\arrow[r] &\Ind(\sD)^{\omega_1}
\end{tikzcd}
\]
to obtain a factorization
\[ \sC \to \sD \stackrel{\to}\times_{\Ind(\sD)^{\omega_1}} \Ind(\sC)^{\omega_1} \to \sD \]
of $F$ into a fully faithful exact functor and a split Verdier projection with kernel $ \Ind(\sC)^{\omega_1}$.
%and observe that we have
Moreover, we have
\[
\K(\sD \odot^\sC_{\Ind(\sD)^{\omega_1}} \Ind(\sC)^{\omega_1}) \simeq \Cofib(\K(\sC) \to \K(\sD)).
\]
Hence we may take $\mrm{Cone}(F)$ to be $\sD \odot^\sC_{\Ind(\sD)^{\omega_1}} \Ind(\sC)^{\omega_1}$ and 
the functor $\sD \to \mrm{Cone}(F)$ to be the composite of 
\[
\sD \hookrightarrow \sD \stackrel{\to}\times_{\Ind(\sD)^{\omega_1}} \Ind(\sC)^{\omega_1} \to \sD \odot^\sC_{\Ind(\sD)^{\omega_1}} \Ind(\sC)^{\omega_1}.\qedhere
\]
\end{proof}

\sssec{} 
The next important step is to show that, along with the suspension and loops functors, the unit equivalence
\[
\eta: x \stackrel{\simeq}{\to} \Omega\Sigma x
\]
can be categorified. 

\begin{thm}\label{thm:unit_loopssuspension}
There exists a natural functor 
\[
\sC \to \Gamma\mrm{Calk}(\sC)
\]
that induces the unit equivalence 
$\K(\sC) \stackrel{\simeq}\to \Omega \Sigma \K(\sC)$
under the equivalences from Proposition~\ref{prop:categorifying_suspensions} and Theorem~\ref{thm:categorifying_loops}.
\end{thm}
\begin{proof}
The natural functors
 \[ \Ind(\sC)^{\omega_1} \to F^q\Ind(\sC)^{\omega_1},\quad X \mapsto X^{\oplus\mbf{N}} \xrightarrow{\id} X^{\oplus\mbf{N}} \to 0 \to 0 \to \ldots \]
 and
 \[ \Ind(\sC)^{\omega_1} \to F\Ind(\sC)^{\omega_1},\quad X \mapsto X^{\oplus\mbf{N}} \xrightarrow{\simeq} X \oplus X^{\oplus\mbf{N}} \xrightarrow{\mrm{pr}} X \xrightarrow{\id} X \xrightarrow{\id} \ldots \]
 become naturally equivalent after composition with $\mrm{gr}$ since both filtrations have associated graded $(X^{\oplus\mbf{N}},0,\Sigma X^{\oplus\mbf{N}},0,0,\ldots)$.
 Hence we obtain a natural transformation
 \[ \sigma \colon \Ind(-)^{\omega_1} \Rightarrow B^t\Ind(-)^{\omega_1}. \]
 In the sequel, denote by $\mrm{S}\sC$ the Verdier quotient of $\Ind(\sC)^{\omega_1}$ by $\sC$ (note that this differs from $\mrm{Calk}(\sC)$ by an idempotent completion).
 Consider the composition
 \[ \tau \colon \Ind(-)^{\omega_1} \overset{\sigma}{\Rightarrow} B^t\Ind(-)^{\omega_1} \Rightarrow B^t\mrm{S}\sC. \]
 Now observe that the composition
 \[ \tau_0 \colon \id \overset{\yo}{\Rightarrow} \Ind(-)^{\omega_1} \overset{\tau}{\Rightarrow} B^t\mrm{S}(-) \]
 factors over the full subfunctor $B^q\mrm{S}(-)$ of $B^t\mrm{S}(-)$.
 This gives rise to the commutative diagram
 \[\begin{tikzcd}
  \sC\ar[r, "\yo"]\ar[d, "\tau_0"] & \Ind(\sC)^{\omega_1}\ar[r]\ar[d, "\tau"] & \mrm{S}\sC\ar[d, "\overline{\tau}"] \\
  B^q\mrm{S}\sC\ar[r]\ar[d, "\id"] & B^t\mrm{S}\sC\ar[r]\ar[d] & B^t\mrm{S}\sC / B^q\mrm{S}\sC\ar[d] \\
  B^q\mrm{S}\sC\ar[r]\ar[d, "\top"] & B\mrm{S}\sC\ar[r]\ar[d, "\top"] & B\mrm{S}\sC / B^q\mrm{S}\sC\ar[d, "\top"] \\
  F^q\mrm{S}\sC\ar[r] & F\mrm{S}\sC\ar[r] & F\mrm{S}\sC / F^q\mrm{S}\sC
 \end{tikzcd}\]
 in which each row is a Karoubi--Verdier sequence.
 Our goal is to show that the composite
 \[ \sC \xrightarrow{\tau_0} B^q\mrm{S}\sC \to \mrm{Cone}(\Delta_{\mrm{S}\sC}) \]
 gives the required natural transformation.
 
 The sequence
 \[ \K(B^t\mrm{S}\sC / B^q\mrm{S}\sC) \to \K( B\mrm{S}\sC / B^q\mrm{S}\sC ) \to \K( F\mrm{S}\sC / F^q\mrm{S}\sC ) \]
 is a fiber sequence by \cite[Corollary~3.17]{KasprowskiWinges}.
 Since $\K(\Ind(\sC)^{\omega_1}) \simeq 0$, the induced map $\K(\tau)$ lifts to the fiber of $\K(B\mrm{S}\sC) \to \K(F\mrm{S}\sC)$ (which is also trivial by \cite[Lemma~3.10]{KasprowskiWinges}).
 Consequently, we obtain a commutative diagram
 \[\begin{tikzcd}
  \K(\sC)\ar[r]\ar[d, "e"] & \K(\Ind(\sC)^{\omega_1})\ar[r]\ar[d, "\simeq"] & \K(\mrm{S}\sC)\ar[d, "\K(\overline{\tau})"] \\
  \Fib\big(\K(B^q\mrm{S}\sC) \xrightarrow{\K(\top)} \K(F^q\mrm{S}\sC)\big)\ar[r] & 0\ar[r] & \K(B^t\mrm{S}\sC / B^q\mrm{S}\sC)
 \end{tikzcd}\]
 Since $\pi_0\K^\mrm{cn}(\mrm{S}\sC)$ is a quotient of $\pi_0\K^\mrm{cn}(\Ind(\sC)^{\omega_1})$, we have $\pi_0\K^\mrm{cn}(\mrm{S}\sC) = 0$.
 Therefore, \cite[Proposition~3.21]{KasprowskiWinges} implies that the composite functor
 \[ \beta \colon B^t\mrm{S}\sC/B^q\mrm{S}\sC \xrightarrow{\bot} F\mrm{S}\sC \to \mrm{S}\sC \]
 is an equivalence, where the latter functor takes the colimit.
 Since $\overline{\tau}$ is evidently a section to $\beta$, we conclude that $\K(\overline{\tau})$ is an equivalence.
 This shows that we have an equivalence
 \[ e \colon \K(\sC) \xrightarrow{\sim} \Fib\big(\K(B^q\mrm{S}\sC) \xrightarrow{\K(\top)} \K(F^q\mrm{S}\sC)\big) \]
 such that the induced map $\K(\sC) \to \K(B^q\sC)$ is given by $\K(\tau_0)$.
 Since $\top$ admits the section $\Delta$, we may canonically identify
 \[ \Fib\big(\K(B^q\mrm{S}\sC) \xrightarrow{\K(\top)} \K(F^q\mrm{S}\sC)\big) \simeq \Cofib\big( \K(F^q\mrm{S}\sC) \xrightarrow{\K(\Delta)} \K(B^q\mrm{S}\sC)\big), \]
 and under this identification we may recover $e$ as the composite
 \[ \K(\sC) \xrightarrow{\K(\tau_0)} \K(B^q\sC) \to \Cofib\big( \K(F^q\mrm{S}\sC) \xrightarrow{\K(\Delta)} \K(B^q\mrm{S}\sC)\big). \]
 Applying Proposition~\ref{prop:categoifying_cofiber}, we finally conclude that the composite
 \[ \sC \xrightarrow{\tau_0} B^q\mrm{S}\sC \to \mrm{Cone}(\Delta_{\mrm{S}\sC}) \]
 induces an equivalence upon application of $\K$.
 Since $\K$ is invariant under idempotent completion, it follows from Theorem~\ref{thm:categorifying_loops} that the natural functor $\mrm{Cone}(\Delta_{\mrm{S}\sC}) \to \mrm{Cone}(\Delta_{\mrm{Calk}(\sC)})$ induces an equivalence after applying $\K$, and we obtain the desired natural transformation $u:\id \Rightarrow \Gamma \mrm{Calk}$ as the composite
 \[ \id \overset{\tau_0}{\Rightarrow} B^q\mrm{S}(-) \Rightarrow \mrm{Cone}(\Delta_{\mrm{S}(-)}) \Rightarrow  \mrm{Cone}(\Delta_{\mrm{Calk}(-)}) \simeq \Gamma \mrm{Calk}(-). \]
 
 Finally, we observe that the induced natural transformation $\K \to \Omega\Sigma\K$ is either homotopic 
 to the unit map $\eta_{\K}$ or to $-\eta_{\K}$ by Corollary~\ref{cor:autoequivalencesKth}. 
 If the former is the case, we are done. Otherwise, $\Sigma u$ will induce the unit map. 
\end{proof}

\begin{rem}\label{rem:generalizing_categorification}
We note that most of the constructions presented above do not use any special properties of K-theory. 
Propositions~\ref{prop:categorifying_suspensions}~and~\ref{prop:categoifying_cofiber} hold for arbitrary 
localizing invariants in place of K-theory. 
Theorem~\ref{thm:categorifying_loops} as well as the proof of the equivalence in 
Theorem~\ref{thm:unit_loopssuspension} are valid for any finitary localizing invariant. 
In particular, all of these apply to the universal localizing invariant 
$\mathcal{U}_\mrm{loc}:\Cat^\perf_\infty \to \Motnc$.

The fact that the map $u:\sC \to \Gamma \mrm{Calk}(\sC)$ categorifies $\eta$ also holds in this generalised setting, 
but a more involved argument is required. 
\end{rem}

\sssec*{Proof of Theorem~\ref{thm:cat_spectrum}}
Given an arbitrary spectrum $M$, Theorem~\ref{thm:DM} and Lemma~\ref{lem:P1} allow us to identify $M$ with a sequential colimit:
\begin{equation}\label{eq:derivative_colimit}
M \simeq \colim\left( \widetilde\K(\Spt^\omega;M) \to \Omega\widetilde\K(\Spt^\omega;\Sigma M) \to \Omega^2\widetilde\K(\Spt^\omega;\Sigma^2 M) \to \ldots \right). 
\end{equation}
Our goal is to realize this sequential colimit as the K-theory of some stable $\infty$-category.
To do this, we begin with an elementary observation concerning sequential colimits.

\begin{lem}\label{lem:seq_colim_cofibre}
 Consider a diagram $x_0 \xrightarrow{\alpha_0} x_1 \xrightarrow{\alpha_1} x_2 \to \ldots$ in a stable $\infty$-category.
 Suppose that each $\alpha_n$ admits a factorization
 \[ \alpha_n \colon x_n \xrightarrow{\phi_n} x_n' \xrightarrow{\psi_n} x_n'' \xrightarrow{\beta_n} x_{n+1}, \]
 where $\phi_n$ and $\psi_n$ are equivalences.
 
 Then $\colim_n x_n$ is equivalent to the cofiber of the map
 \[ \bigoplus_{n \in \mbf{N}} x_n \oplus x_n'' \to \bigoplus_{n \in \mbf{N}} x_n \oplus x_n' \]
 represented by the following diagram:
 \[\begin{tikzcd}[column sep=4em]
  x_0\ar[r, phantom, "\oplus"]\ar[d, "\id"]\ar[dr, "-\phi_0"] & x_0''\ar[r, phantom, "\oplus"]\ar[d, "\psi_0^{-1}"]\ar[dr, "-\beta_0"] & x_1\ar[r, phantom, "\oplus"]\ar[d, "\id"]\ar[dr, "-\phi_1"] & x_1''\ar[r, phantom, "\oplus"]\ar[d, "\psi_1^{-1}"]\ar[dr, "-\beta_1"] & x_2\ar[r, phantom, "\oplus"]\ar[d, "\id"]\ar[dr, "-\phi_2"] & x_2''\ar[r, phantom, "\oplus"]\ar[d, "\psi_2^{-1}"]\ar[dr, "-\beta_2"] & \ldots \\
  x_0\ar[r, phantom, "\oplus"] & x_0'\ar[r, phantom, "\oplus"] & x_1\ar[r, phantom, "\oplus"] & x_1'\ar[r, phantom, "\oplus"] & x_2\ar[r, phantom, "\oplus"] & x_2'\ar[r, phantom, "\oplus"] & \ldots
 \end{tikzcd}\]
\end{lem}
\begin{proof}
Set $\gamma_n := \psi_n\phi_n \colon x_n \to x_n''$.
 The map in question is equivalent to the map $\bigoplus_{n \in \mbf{N}} x_n \oplus x_n'' \to \bigoplus_{n \in \mbf{N}} x_n \oplus x_n'$ represented by the diagram
 \[\begin{tikzcd}[column sep=4em]
  x_0\ar[r, phantom, "\oplus"]\ar[d, "\id"]\ar[dr, "-\gamma_0"] & x_0''\ar[r, phantom, "\oplus"]\ar[d, "\id"]\ar[dr, "-\beta_0"] & x_1\ar[r, phantom, "\oplus"]\ar[d, "\id"]\ar[dr, "-\gamma_1"] & x_1''\ar[r, phantom, "\oplus"]\ar[d, "\id"]\ar[dr, "-\beta_1"] & x_2\ar[r, phantom, "\oplus"]\ar[d, "\id"]\ar[dr, "-\gamma_2"] & x_2''\ar[r, phantom, "\oplus"]\ar[d, "\id"]\ar[dr, "-\beta_2"] & \ldots \\
  x_0\ar[r, phantom, "\oplus"] & x_0''\ar[r, phantom, "\oplus"] & x_1\ar[r, phantom, "\oplus"] & x_1''\ar[r, phantom, "\oplus"] & x_2\ar[r, phantom, "\oplus"] & x_2''\ar[r, phantom, "\oplus"] & \ldots
 \end{tikzcd}\]
 The cofiber of this map is equivalent to the colimit of the diagram
 \[ x_0 \xrightarrow{\gamma_0} x_0'' \xrightarrow{\beta_0} x_1 \xrightarrow{\gamma_1} x_1'' \xrightarrow{\beta_1} x_2 \xrightarrow{\gamma_2} \ldots, \]
 which is evidently equivalent to $\colim_n x_n$.
\end{proof}

Recall that $\widetilde\K(\Spt^\omega;\Sigma^n M)$ is defined as the cofiber of the map 
\[
\K(\mrm{End}(\Spt^\omega;0)) \stackrel{\K(i_{n-1})}\to  \K(\mrm{End}(\Spt^\omega;\Sigma^{n} M)).
\]
Considering the commutative diagram
\[\begin{tikzcd}
 \mrm{End}(\Spt^\omega;\Sigma^n M)\ar[r, "p_n"]\ar[d] & \mrm{End}(\Spt^\omega;0)\ar[d] \\
 \mrm{End}(\Spt^\omega;0)\ar[r, "i_n"] & \mrm{End}(\Spt^\omega;\Sigma^{n+1} M)
\end{tikzcd}\] 
we may identify the structure maps in the sequential colimit (\ref{eq:derivative_colimit}) with the compositions
\[ \widetilde\K(\Spt^\omega;\Sigma^n M) \xrightarrow{\eta_n} \Omega\Sigma\widetilde\K(\Spt^\omega;\Sigma^n M) \xrightarrow{v_n} \Omega\Cofib(\K(p_n)) \xrightarrow{f_n} \Omega\Cofib(\K(i_n)) = \Omega\widetilde\K(\Spt^\omega;\Sigma^{n+1} M).\]
Here, $\eta_n$ is the unit map $\id \xrightarrow{\sim} \Omega\Sigma$, the map $v_n$ is the inverse of the map induced on the (horizontal) cofibers of the pullback diagram 
\[\begin{tikzcd}
 \K(\Spt^\omega;\Sigma^n M)\ar[r, "\K(p_n)"]\ar[d] & \K(\Spt^\omega;0)\ar[d] \\
 \widetilde\K(\Spt^\omega;\Sigma^n M)\ar[r] & 0
\end{tikzcd}\]
and $f_n$ is the canonical map induced on cofibers.
Note already that we will be able to apply Lemma~\ref{lem:seq_colim_cofibre} to this situation.

First, apply Proposition~\ref{prop:categoifying_cofiber} to the exact functor $\mrm{End}(\Spt^\omega;0) \stackrel{i_{n-1}}\to \mrm{End}(\Spt^\omega;\Sigma^n M)$ to obtain a stable $\infty$-category $\widetilde{\mrm{End}}(\Spt^\omega;\Sigma^n M)$ satisfying
\[ \K(\widetilde{\mrm{End}}(\Spt^\omega;\Sigma^n M)) \simeq \widetilde\K(\Spt^\omega;\Sigma^n M). \]
By Theorem~\ref{thm:unit_loopssuspension}, there exists an exact functor
\[ \phi_n \colon \widetilde{\mrm{End}}(\Spt^\omega;\Sigma^n M) \to \Gamma\mrm{Calk}(\widetilde{\mrm{End}}(\Spt^\omega;\Sigma^n M)) \]
such that $\K(\phi_n)$ is identified with $\eta_n$.
Now observe that the commutative square
\[\begin{tikzcd}
 \mrm{End}(\Spt^\omega;\Sigma^n M)\ar[r, "p_n"]\ar[d] & \mrm{End}(\Spt^\omega;0)\ar[d] \\
 \widetilde{\mrm{End}}(\Spt^\omega;\Sigma^n M)\ar[r] & 0
\end{tikzcd}\]
induces via Proposition~\ref{prop:categoifying_cofiber} an exact functor
\[ \psi_n \colon \mrm{Cone}(p) \to \mrm{Calk}(\widetilde{\mrm{End}}(\Spt^\omega;-\otimes \Sigma^n M)) \]
such that $\K(\Gamma\psi_n)$ is identified with $v_n^{-1}$.
Appealing to Proposition~\ref{prop:categoifying_cofiber} once again, we obtain an exact functor
\[ \beta_n \colon \mrm{Cone}(p) \to \widetilde{\mrm{End}}(\Spt^\omega;\Sigma^{n+1} M) \]
such that $\K(\beta_n)$ is identified with $f_n$.

Now consider the exact functor $F$ represented by the diagram
\[\begin{tikzcd}
  \widetilde{\mrm{End}}(\Spt^\omega;M)\ar[r, phantom, "\oplus"]\ar[d, "\id"]\ar[dr, "\Sigma \circ \phi_0"] & \Gamma\mrm{Cone}(p_0)\ar[r, phantom, "\oplus"]\ar[d, "\Gamma\psi_0"]\ar[dr, "\Sigma \circ \Gamma\beta_0"] & \widetilde{\mrm{End}}(\Spt^\omega;\Sigma M)\ar[r, phantom, "\oplus"]\ar[d, "\id"]\ar[dr, "\Sigma \circ \phi_1"] & \Gamma \mrm{Cone}(p_1)\ar[r, phantom, "\oplus"]\ar[d, "\Gamma\psi_1"]\ar[dr, "\Sigma \circ \Gamma\beta_1"] & \ldots \\
  \widetilde{\mrm{End}}(\Spt^\omega;M)\ar[r, phantom, "\oplus"] & \Gamma\mrm{Calk}(\widetilde{\mrm{End}}(\Spt^\omega;M))\ar[r, phantom, "\oplus"] & \widetilde{\mrm{End}}(\Spt^\omega;\Sigma M)\ar[r, phantom, "\oplus"] & \Gamma\mrm{Calk}(\widetilde{\mrm{End}}(\Spt^\omega;\Sigma M))\ar[r, phantom, "\oplus"] & \ldots
 \end{tikzcd}\]
 Since K-theory preserves arbitrary direct sums and the suspension functors $\Sigma$ induce $-\id$ on K-theory, the preceding discussion in combination with Lemma~\ref{lem:seq_colim_cofibre} shows that
 \[ \K(\mrm{Cone}(F)) \simeq M \]
 as desired.
 
 The functoriality of $\mrm{Cone}$ and $\Gamma$ together with the naturality of the functors $\phi_n$ imply that the construction of $\mrm{Cone}(F)$ refines to a functor $\sC_{(-)} : \Spt \to \Cat^{\perf}_{\infty}$ such that $\K \circ \sC_{(-)} \simeq \id$.

\begin{rem}\label{rem:rings}
Theorem~\ref{thm:cat_spectrum} trivially implies that every abelian group is $\K_0(\sC)$ for some $\sC \in \Cat^\perf_\infty$. 
A more interesting question concerning categorifying (commutative) {\it rings}, rather than abelian groups, remains open. 
Given a ring $R$ one wants to find a stably symmetric monoidal small $\infty$-category $\sC$ with 
$\K_0(\sC) \simeq R$. 
The strongest result in this direction is proved in \cite{levy2023categorifying}.
\end{rem}

\section{Bounded t-structures and the theorem of the heart}\label{sec:boundedtstructures}

\begin{defn}\label{defn:t-str}
Let $\sC$ be a stable $\infty$-category. 
A pair of full subcategories $(\sC_{\ge 0}, \sC_{\le 0})$ is said to be a {\bf t-structure} if the following 
axioms hold:
\begin{enumerate}
\item $\Sigma \sC_{\ge 0} \subset \sC_{\ge 0}$ and $\Omega \sC_{\le 0} \subset \sC_{\le 0}$;
\item for $x \in \sC_{\ge 0}$ and $y\in \sC_{\le 0}$ we have $\pi_0\Maps_\sC(\Sigma x, y) = 0$;
\item for $x\in \sC$ there exists a fiber sequence 
\[
x_{\ge 1} \to x \to x_{\le 0}
\]
such that $x_{\ge 1} \in \Sigma \sC_{\ge 0}$ and $x_{\le 0} \in \sC_{\le 0}$.
\end{enumerate}

The {\bf heart} of a t-structure on $\sC$ is $\sC^\heartsuit := \sC_{\ge 0} \cap \sC_{\le 0}$. 

An exact functor $F:\sC \to \sD$ of stable $\infty$-categories equipped with t-structures is called t-exact 
if $F(\sC_{\ge 0}) \subset \sD_{\ge 0}$ and $F(\sC_{\le 0}) \subset \sD_{\le 0}$.

We set $\sC_{\ge n} = \Sigma^n \sC_{\ge 0}$ and $\sC_{\le n} = \Sigma^n \sC_{\le 0}$.

A t-structure on $\sC$ is called bounded if 
$\sC = \bigcup\limits_{n\in \mbf{N}} (\sC_{\ge -n} \cap \sC_{\le n})$.
\end{defn}

\subsubsection{Basic properties}
\begin{enumerate}
\item It formally follows from the definition that the choice of a fiber sequence in (3) is unique up to 
equivalence and is functorial in $x$. 
\item The subcategories $\sC_{\ge 0}$ and $\sC_{\ge 0}$ determine each other as the corresponding 
categorical orthogonals, that is, we have:
\[
\sC_{\ge 0} = \{x \in \sC : \pi_0\Maps_\sC(\Sigma x,y) = 0 \text{ for any } y \in \sC_{\le 0} \},
\]
\[
\sC_{\le 0} = \{x \in \sC : \pi_0\Maps_\sC(\Sigma y,x) = 0 \text{ for any } y \in \sC_{\ge 0} \}.
\]
\item The heart of a t-structure is always an abelian category. 
\item A stable $\infty$-category with a bounded t-structure is automatically idempotent complete 
\cite[Corollary~2.14]{AGH}. 
\end{enumerate}

\begin{exam}\label{exam:t-spectra}
The $\infty$-category of spectra admits a t-structure with $\Spt_{\ge 0}$ being the subcategory 
of connective spectra and $\Spt_{\le 0}$ the subcategory of coconnective spectra. The fact that 
the axiom (3) is satisfied reflects  
the existence of Postnikov towers of spectra. This t-structure {\bf is not} bounded. The heart of this 
t-structure consists of Eilenberg--MacLane spectra of abelian groups, and is equivalent to $\mathcal{A}b$. 
\end{exam}

\begin{exam}
Generalizing the previous example, for a connective ring spectrum $R$, 
the $\infty$-category of right $R$-module spectra $\Mod_R$ admits a t-structure with $(\Mod_{R})_{\ge 0}$ 
being the subcategory of connective $R$-module spectra. 
\end{exam}

\begin{exam}\label{exam:derived-category}
Given a small abelian category $A$, the bounded derived category $\sD^b(A)$ admits a bounded t-structure 
with $\sD^b(A)_{\ge 0}$ consisting of complexes $x$ with $H_i(x) = 0$ for $i<0$ and similarly for 
$\sD^b(A)_{\le 0}$. The heart of this t-structure is equivalent to $A$.
\end{exam}

The following result shows that the bounded t-structure of Example~\ref{exam:derived-category} is in some sense 
a universal one.

\begin{prop}[{\cite[Proposition~3.26]{AGH}}]\label{prop:derived_universal}
For any small stable $\infty$-category $\sC$ equipped with a bounded t-structure, 
there is a natural t-exact functor 
\[
\sD^b(\sC^\heartsuit) \to \sC
\]
that induces an equivalence on the hearts of the t-structures.
\end{prop}

\ssec{K-theoretic properties of t-structures}
Proposition~\ref{prop:derived_universal} yields a canonical map 
\[
\K(\sD^b(\sC^\heartsuit)) \to \K(\sC)
\]
for any small stable $\infty$-category $\sC$ endowed with a bounded t-structure. 
In \cite{AGH} Antieau, Gepner, and Heller formulated the following conjectures:

\begin{conj}[Conjecture~B]\label{conj:Kvanishing}
Let $\sC$ be a small stable $\infty$-category equipped with a bounded t-structure. 
Then the K-theory spectrum of $\sC$ is connective.
\end{conj}

\begin{conj}[Conjecture~C]\label{conj:thmoftheheart}
Let $\sC$ be a small stable $\infty$-category equipped with a bounded t-structure. 
Then the map 
\[
\K(\sD^b(\sC^\heartsuit)) \to \K(\sC)
\]
is an equivalence of spectra.
\end{conj}

\begin{rem}\label{rem:history_heart}
Both conjectures were proved in \cite{AGH} for t-structures whose heart is 
a {\it noetherian} abelian category. Besides, they showed that $\pi_{-1}\K(\sC) = 0$ without any extra 
assumptions on the t-structure. In \cite{Schlichting2006} a version of 
Conjecture~\ref{conj:Kvanishing} was investigated for $\sC = \sD^b(A)$ and proved under the same assumptions. 

It is also worth pointing out that Conjecture~\ref{conj:thmoftheheart} 
is true for the {\it connective} version of K-theory. This is known as the theorem of the heart and was 
proved by Clark Barwick \cite{Barwick_heart} and a triangulated version of it had been proved even before 
in the works of Amnon Neeman~\cite{Neeman_HEART_IIIA,Neeman_heart_IIIB,Neeman_heart_334}. 
In particular, Conjecture~\ref{conj:thmoftheheart} would follow from Conjecture~\ref{conj:Kvanishing}.

However, Conjecture~\ref{conj:Kvanishing} in its full generality turned out to be false. 
In \cite{Neeman2021} Neeman constructed an explicit stable $\infty$-category $\sA$ with a t-structure such that 
$\pi_{-2}\K(\sA)$ is not zero. Until now Conjecture~\ref{conj:thmoftheheart} has remained open.
\end{rem}

\ssec{Chromatic obstructions to bounded t-structures}\label{ssec:chromatic_tstr}

In this paper we show that Conjecture~\ref{conj:thmoftheheart} is false. 
To do this, we will need to investigate some general {\it chromatic} properties of the spectra 
$\K(\sC)$ for small stable 
$\infty$-categories $\sC$. First recall some basic definitions.

\begin{defn}\label{defn:local-spectra}
Let $E, X$ be spectra. 

\begin{enumerate}
\item 
$X$ is called $E$-{\bf acyclic} if $X \otimes E \simeq 0$. 
\item 
$X$ is called $E$-{\bf local} if it is orthogonal to all $E$-acyclic spectra, i.e. for any $E$-acyclic spectrum $Y$
$\Maps_\Spt(Y,X) = 0$.
\end{enumerate}
\end{defn}

%\begin{defn}
%Consider the Morava K-theory spectra $K(n)$ for a fixed prime $p$. 
%For a spectrum $X$ the maximal integer $n$ such that $X \otimes K(n)$ is not trivial is called the {\bf chromatic 
%height} of $X$.
%\end{defn}

\begin{exam}\label{exam:modules_local}
Assume $E$ is a ring spectrum. 
Then any $E$-module spectrum $X$ is $E$-local. To see this, observe that for any $E$-acyclic spectrum $Y$ 
\[
\Maps_\Spt(Y,X) \simeq \Maps_E(Y \otimes E,X) \simeq \Maps_E(0,X) \simeq 0.
\]
\end{exam}

\begin{exam}\label{exam:connective_local}
Let $E$ be a connective ring spectrum such that $\pi_0R = \Z$. Then 
any (eventually) connective spectrum $X$ is $E$-local. This is an instance of a much more general 
\cite[Theorem~3.1]{Bousfield1979}.
\end{exam}

\begin{exam}\label{exam:nonlocal-Morava}
Fix any prime $p$. A theorem of Mitchell \cite{Mitchell1990} states that 
the Morava K-theory spectra $K(n)$ are $\K(\Z)$-acyclic for $n\ge 2$. In particular, the Morava K-theories 
are all examples of spectra that are not $\K(\Z)$-local.
\end{exam}

Immediately from Example~\ref{exam:nonlocal-Morava} we can observe that the K-theory spectra obtained from 
derived categories are quite special from the chromatic perspective.

\begin{prop}\label{prop:locality_tstructure}
Let $A$ be a small abelian category.
Then $\K(\sD^b(A))$ is a module over the K-theory of the integers $\K(\Z)$. 
\end{prop}
\begin{proof}
The stable $\infty$-category $\sD^b(A)$ is $\Z$-linear. In particular, it admits an action by 
$\sD^b(\Z)$. This induces the desired action of $\K(\Z)$ on $\K(\sD^b(A))$.
\end{proof}

\begin{cor}\label{cor:locality_tstructure}
Let $\sC$ be a small stable $\infty$-category equipped with a bounded t-structure. 
If Conjecture~\ref{conj:thmoftheheart} holds for $\sC$, then $\K(\sC)$ is $\K(\Z)$-local.
\end{cor}

\sssec{}\label{sssec:finmodules}
For a small additive $\infty$-category $\sA$ consider the $\infty$-category of {\it additive} spectrum-valued presheaves on $\sA$:
\[
\hat{\sA} := \Fun_{\times}(\sA^\mrm{op}, \Spt).
\]
We define $\sA^\mrm{fin}$ to be the smallest stable subcategory of $\hat{\sA}$ that contains the image of the 
Yoneda embedding
\[
\yo : \sA \to \hat{\sA}.
\]
The functor $\sA \to \sA^\mrm{fin}$ is an initial additive functor from $\sA$ into a small stable 
$\infty$-category \cite[Remark~2.1.18]{dgm-elden-vova}. 
The $\infty$-category $\sA^\mrm{fin}$ is idempotent complete if $\sA$ is \cite[Theorem~2.2.9(2)]{dgm-elden-vova}. 

\begin{thm}\label{thm:weight_local}
Let $\sA$ be an idempotent complete small additive $\infty$-category. 
Then $\K(\sA^\mrm{fin})$ is a $\K(\Z)$-local spectrum. 
\end{thm}
\begin{proof}
Assume first that $\sA$ is a classical additive category. Then it is canonically $\Z$-linear and 
so is $\sA^\mrm{fin}$. Hence $\K(\sA^\mrm{fin})$ is a $\K(\Z)$-local spectrum. 

Now for a general $\sA$ consider the initial functor into a classical additive category
\[
\sA \to \mrm{h}\sA.
\] 
This induces a functor 
\[
\sA^\mrm{fin} \to (\mrm{h}\sA)^\mrm{fin}
\]
and a map 
\[
\K(\sA^\mrm{fin}) \to \K((\mrm{h}\sA)^\mrm{fin}).
\]
This map has connective fiber by \cite[Theorem~4.3]{Sosnilo_2019} 
(see also \cite[Corollary~5.3.12]{sosnilo2022weighted}), so it is $\K(\mbf{Z})$-local by Example~\ref{exam:connective_local}.
Since $\K(\mbf{Z})$-local spectra form a stable subcategory, the theorem follows.
\end{proof}

\begin{rem}\label{rem:weight-additives}
The construction $\sA \mapsto \sA^{\mrm{fin}}$ is central in the theory of weight structures. 
In fact, $\sA^{\mrm{fin}}$ is canonically endowed with a bounded weight structure and 
any small stable $\infty$-category equipped with a bounded weight structure is canonically equivalent to one of 
this form. Thus, in the spirit of Corollary~\ref{cor:locality_tstructure}, we can reformulate 
Theorem~\ref{thm:weight_local} into saying that the K-theory of any stable $\infty$-category with a bounded 
weight structure is a $\K(\Z)$-local spectrum.

To further illustrate that $\K(\sA^{\mrm{fin}})$ is $\K(\mbf{Z})$-local when $\sA$ is a classical additive category, let us also remark that $\sA^{\mrm{fin}}$ is equivalent to the Dwyer--Kan localization of the ordinary category of bounded chain complexes at the collection of chain homotopy equivalences.
For example, this can be seen by observing that the full subcategories of chain complexes concentrated in non-negative (respectively non-positive) degrees define a bounded weight structure on this localization.
\end{rem}

\section{Counterexample to the nonconnective theorem of the heart}\label{sec:toh_is_false}

Recall from (\ref{sssec:finmodules}) that we have an adjunction 
\[
\begin{tikzcd}
(-)^\mrm{fin} : \Cat^\mrm{padd}_\infty \arrow[r,shift left=.5ex]
&
\Cat^\perf_\infty : U \arrow[l,shift left=.5ex]
\end{tikzcd}
\]
Given a small stable $\infty$-category $\sC \in \Cat^\perf_\infty$, consider the counit map
\[
L : \sC^{\mrm{fin}} \to \sC.
\]
For any stable $\infty$-category $\sD$ this induces a fully faithful embedding
\[
\Funex(\sC, \sD) \hookrightarrow \Funex(\sC^{\mrm{fin}}, \sD) \simeq \Fun_\times(\sC, \sD).
\]
Thus $L$ is a Verdier localization and its kernel $\sA c(\sC)$ is generated by objects of the form 
\begin{equation}\label{eq:generators}
\Cofib(\yo(x)/\yo(y) \to \yo(x/y))
\end{equation}
for all maps $y\to x$ in $\sC$ (see also \cite[Proposition~4.22]{Klemenc_2022}).

\begin{thm}\label{thm:acyclic_tstructure}
For any small stable $\infty$-category $\sC$ there is a bounded t-structure on $\sA c(\sC)$ whose heart 
contains all the objects (\ref{eq:generators}).
\end{thm}
\begin{proof}
Following \cite[Definition~4.20]{patchkoria2023adams}, for an additive $\infty$-category $\sA$ we denote by 
$A_\infty^\omega(\sA)$ the smallest subcategory 
of $\Fun_\times(\sA^{\mrm{op}}, \sS)$ containing the image of the Yoneda embedding 
and closed under finite colimits. This is a prestable $\infty$-category and the functor 
$\sA \to A_\infty^\omega(\sA)$ satisfies the following universal property: 
any additive functor $\sA \to \sB$ where $\sB$ is an additive $\infty$-category with finite colimits factors 
essentially uniquely through a right exact functor $A_\infty^\omega(\sA) \to \sB$. 
There is a canonical equivalence
\[
\sA^\mrm{fin} \simeq \mrm{SW}(A_\infty^\omega(\sA))
\]
which follows from the universal properties of both sides (see also \cite[Remark~2.1.18]{dgm-elden-vova}).

Now consider the case $\sA = \sC$. Since $\sC$ admits finite limits, $A_\infty^\omega(\sC)$ does, 
too, by \cite[Theorem~4.26]{patchkoria2023adams}. Hence, by \cite[Proposition~C.1.2.9]{SAG} there is a t-structure on $\sC^\mrm{fin}$. 
Moreover, the induced functor
\[
A_\infty^\omega(\sC) \to \sC
\]
is left exact by \cite[Theorem~4.49]{patchkoria2023adams}. 
By \cite[Proposition~C.3.2.1]{SAG} the induced functor 
\[
\sC^\mrm{fin} \to \sC
\]
is t-exact where $\sC$ is endowed with the trivial t-structure: $\sC_{t\ge 0} = \sC$ and 
$\sC_{t\le 0} = \{0\}$. 
In particular, the t-structure on $\sC^\mrm{fin}$ restricts to $\sA c (\sC)$. Note that 
all objects of the form 
(\ref{eq:generators}) have discrete values on objects of $\sC \subset \sC^\mrm{fin}$ 
and thus belong to the heart of the t-structure. Since $\sA c(\sC)$ is generated by these objects under finite 
colimits and finite limits, we conclude that the t-structure on $\sA c(\sC)$ is bounded.
\end{proof}

\begin{rem}\label{rem:exactcat_ac}
The construction of a t-structure on $\sA c(\sC)$ is a direct adaptation of a construction of Neeman 
done in the setting of exact categories \cite[Section~1]{Neeman2021}. 

The t-structure of Theorem~\ref{thm:acyclic_tstructure} is also implicit in \cite{Klemenc_2022}: Proposition~3.16 in 
\textit{loc.\ cit.} asserts that the connective cover of any object in $\sA c(\sC)$ also lies in $\sA c(\sC)$, so the 
Postnikov t-structure on spectrum-valued presheaves retricts to a t-structure on $\sA c(\sC)$.
Since $\sA c(\sC)$ is generated by discrete objects (cf.~\cite[Proposition~3.10]{Klemenc_2022}), the resulting 
t-structure is bounded.
\end{rem}

\begin{rem}\label{rem:bounded_objects}
From the proof of Theorem~\ref{thm:acyclic_tstructure} it follows that $\sC^\mrm{fin}$ also 
admits a t-structure. It is, however, not bounded (from above), and $\sA c(\sC)$ is exactly the subcategory of 
t-bounded  objects in $\sC^\mrm{fin}$. Furthermore, one can identify the heart of the t-structure as the 
category of finitely presented additive presheaves of abelian groups\footnote{ 
also known as the {\it Freyd envelope} of $\sC$ \cite[Section~3]{Freyd1966}, \cite[Chapter~5]{Neeman2001}} on $\sC$
(see \cite[Example~4.17]{patchkoria2023adams}).
\end{rem}

\begin{rem}\label{rem:perfect_derived_category}
Theorem~\ref{thm:acyclic_tstructure} was stated and proved by Patchkoria and Pstr\k{a}gowski using a slightly 
different language in \cite[Section~5.2]{patchkoria2023adams}. 
There they define the {\it perfect derived} 
$\infty$-category $\sD^\omega(\sC; H)$ of a stable $\infty$-category $\sC$ associated with an {\it adapted} 
homology theory $H$ on $\sC$. It is a prestable $\infty$-category with finite limits. 
They construct a left and right exact functor of prestable $\infty$-categories 
\[
A_\infty^\omega(\sC) \to \sD^\omega(\sC; H)
\]
which is a localization and whose kernel consists of bounded objects (see \cite[Theorem~5.34(4)]{patchkoria2023adams}). After applying $(-)^\mrm{fin}$ this yields a t-exact Verdier localization functor such 
that the 
kernel consists of t-bounded objects. 
In case $H$ is the trivial homology theory valued in the trivial abelian category, we have $\sD^\omega(\sC; H) = \sC$ and deduce the claim of Theorem~\ref{thm:acyclic_tstructure}.
\end{rem}

Now we can finally show Theorem~\ref{thmX:tohisfalse}.

\sssec{Construction of the counterexample}\label{sssec:counterexample}

For any stable $\infty$-category $\sC$ we have a fiber sequence
\[
\K(\sA c(\sC)) \to \K(\sC^\mrm{fin}) \to \K(\sC).
\]
For any spectrum $M$ we can apply this to $\sC = \sC_M$ from Theorem~\ref{thm:cat_spectrum}. We now have 
a fiber sequence 
\[
\K(\sA c(\sC)) \to \K(\sC^\mrm{fin}) \to M.
\]
Theorem~\ref{thm:weight_local} says that the spectrum $\K(\sC^\mrm{fin})$ is $\K(\Z)$-local. 
Thus, for any $M$ that is not $\K(\Z)$-local, $\K(\sA c(\sC))$ is not $\K(\Z)$-local. 
Such spectra exist: we can pick $M=K(2)$ by Example~\ref{exam:nonlocal-Morava}. 
Since $\sA c (\sC_M)$ is endowed with a bounded t-structure,
Conjecture~\ref{conj:thmoftheheart} fails by Corollary~\ref{cor:locality_tstructure}.

\begin{rem}\label{rem:K_notA1invariant}
The formula 
\[
\sC \mapsto \colim_{\Delta^{\mrm{op}}} \K(\sC \otimes \mbf{Perf}_{\sphere[\Delta^\bullet]})
\]
defines a localizing invariant $\mrm{KH}:\Cat^\perf_\infty \to \Spt$ (see \cite[1.14.2]{vova-invariance} and also 
\cite[Definition~3.13]{Land_2019}). 
By \cite[Corollary~5.1.9]{cisinski2017a1homotopy} (see also \cite[Theorem~A]{Khan_2019}) we have 
\[
\mrm{KH}(\sphere) \simeq \mrm{KH}(\Z) \simeq \mrm{K}(\Z),
\] 
which gives a $\mrm{K}(\Z)$-module structure on $\mrm{KH}(\sC)$ for any $\sC \in \Cat^\perf_\infty$. 
Now the construction in (\ref{sssec:counterexample}) shows that there exists a stable $\infty$-category $\sC$ endowed 
with a bounded t-structure such that $\K(\sC)$ is not equivalent to $\mrm{KH}(\sC)$. 
In particular, the maps in the colimit defining \mrm{KH} are not all equivalences and so, there exists an 
integer $n>0$ such that the map
\[
\mrm{K}(\sC) \to \mrm{K}(\sC \otimes \mbf{Perf}_{\sphere[t_1,\dots, t_n]})
\]
is {\it not} an equivalence. This shows that the $\mathbb{A}^n$-invariance in the sense of 
\cite[Theorem~1.6]{Burklund_2023} does not hold for nonconnective algebraic K-theory.
\end{rem}

\let\mathbb=\mathbf

{\small
\bibliography{references}
}

\parskip 0pt

\end{document}

%% file: preamble.tex
\usepackage[utf8]{inputenc}
\usepackage{url}
\usepackage[hidelinks]{hyperref}
\usepackage{pdfpages}
\usepackage[DIV=10]{typearea}

\usepackage{microtype}
\usepackage{amssymb,amsmath,amsopn,amsxtra,amsthm,amsfonts}
\usepackage{MnSymbol}
\usepackage[mathcal]{euscript}
\let\mathscr\mathcal
\usepackage[bb=px]{mathalfa}

\usepackage[T1]{fontenc}

\newcommand{\yo}{\text{\usefont{U}{min}{m}{n}\symbol{'210}}}

\DeclareFontFamily{U}{min}{}
\DeclareFontShape{U}{min}{m}{n}{<-> udmj30}{}

\usepackage{enumitem}
\setlist[enumerate,1]{label={(\arabic*)},itemsep=\parskip} %,leftmargin=0pt
\setlist[itemize,1]{itemsep=\parskip} %,leftmargin=0pt
\newlist{thmlist}{enumerate}{2}
\setlist[thmlist,1]{label={\em(\roman*)},ref={(\roman*)},%
  itemsep=\parskip,leftmargin=*,align=left}
\setlist[thmlist,2]{label={\em(\alph*)},ref={(\alph*)},%
  itemsep=\parskip,leftmargin=*,align=left,topsep=0.1cm}
\newlist{defnlist}{enumerate}{2}
\setlist[defnlist,1]{label={(\roman*)},ref={(\roman*)},itemsep=\parskip,%
  leftmargin=*,align=left}
\setlist[defnlist,2]{label={(\alph*)},ref={(\alph*)},itemsep=\parskip,%
  leftmargin=*,align=left,topsep=0.1cm}

 \newtheorem*{thm*}{Theorem}
\newtheorem{thmX}{Theorem}
\newtheorem{thm}[subsection]{Theorem}
\newtheorem{cor}[subsection]{Corollary}
\newtheorem{lem}[subsection]{Lemma}
\newtheorem{prop}[subsection]{Proposition}

\newtheorem{conj}[subsection]{Conjecture}

\theoremstyle{definition}
\newtheorem{defn}[subsection]{Definition}

\newtheorem{rem}[subsection]{Remark}

\newtheorem{exam}[subsection]{Example}

\newtheorem{constr}[subsection]{Construction}

\renewcommand{\eqref}[1]{(\ref{#1})}

\numberwithin{equation}{subsection}

\usepackage{tikz}
\usetikzlibrary{matrix}
\usepackage{tikz-cd}

\usepackage{xspace}

\usepackage{xifthen}

\usepackage{xparse}

\usepackage{mathtools}

\newcommand{\nc}{\newcommand}
\nc{\renc}{\renewcommand}
\nc{\ssec}{\subsection}
\nc{\sssec}{\subsubsection}
\nc{\on}{\operatorname}
\nc{\term}[1]{#1\xspace}

\nc{\sA}{\ensuremath{\mathcal{A}}\xspace}
\nc{\sB}{\ensuremath{\mathcal{B}}\xspace}
\nc{\sC}{\ensuremath{\mathcal{C}}\xspace}
\nc{\sD}{\ensuremath{\mathcal{D}}\xspace}
\nc{\sE}{\ensuremath{\mathcal{E}}\xspace}
\nc{\sF}{\ensuremath{\mathcal{F}}\xspace}
\nc{\sG}{\ensuremath{\mathcal{G}}\xspace}
\nc{\sH}{\ensuremath{\mathcal{H}}\xspace}
\nc{\sI}{\ensuremath{\mathcal{I}}\xspace}
\nc{\sJ}{\ensuremath{\mathcal{J}}\xspace}
\nc{\sK}{\ensuremath{\mathcal{K}}\xspace}
\nc{\sL}{\ensuremath{\mathcal{L}}\xspace}
\nc{\sM}{\ensuremath{\mathcal{M}}\xspace}
\nc{\sN}{\ensuremath{\mathcal{N}}\xspace}
\nc{\sO}{\ensuremath{\mathcal{O}}\xspace}
\nc{\sP}{\ensuremath{\mathcal{P}}\xspace}
\nc{\sQ}{\ensuremath{\mathcal{Q}}\xspace}
\nc{\sR}{\ensuremath{\mathcal{R}}\xspace}
\nc{\sS}{\ensuremath{\mathcal{S}}\xspace}
\nc{\sT}{\ensuremath{\mathcal{T}}\xspace}
\nc{\sU}{\ensuremath{\mathcal{U}}\xspace}
\nc{\sV}{\ensuremath{\mathcal{V}}\xspace}
\nc{\sW}{\ensuremath{\mathcal{W}}\xspace}
\nc{\sX}{\ensuremath{\mathcal{X}}\xspace}
\nc{\sY}{\ensuremath{\mathcal{Y}}\xspace}
\nc{\sZ}{\ensuremath{\mathcal{Z}}\xspace}

\nc{\bA}{\ensuremath{\mathbf{A}}\xspace}
\nc{\bB}{\ensuremath{\mathbf{B}}\xspace}
\nc{\bC}{\ensuremath{\mathbf{C}}\xspace}
\nc{\bD}{\ensuremath{\mathbf{D}}\xspace}
\nc{\bE}{\ensuremath{\mathbf{E}}\xspace}
\nc{\bF}{\ensuremath{\mathbf{F}}\xspace}
\nc{\bG}{\ensuremath{\mathbf{G}}\xspace}
\nc{\bH}{\ensuremath{\mathbf{H}}\xspace}
\nc{\bI}{\ensuremath{\mathbf{I}}\xspace}
\nc{\bJ}{\ensuremath{\mathbf{J}}\xspace}
\nc{\bK}{\ensuremath{\mathbf{K}}\xspace}
\nc{\bL}{\ensuremath{\mathbf{L}}\xspace}
\nc{\bM}{\ensuremath{\mathbf{M}}\xspace}
\nc{\bN}{\ensuremath{\mathbf{N}}\xspace}
\nc{\bO}{\ensuremath{\mathbf{O}}\xspace}
\nc{\bP}{\ensuremath{\mathbf{P}}\xspace}
\nc{\bQ}{\ensuremath{\mathbf{Q}}\xspace}
\nc{\bR}{\ensuremath{\mathbf{R}}\xspace}
\nc{\bS}{\ensuremath{\mathbf{S}}\xspace}
\nc{\bT}{\ensuremath{\mathbf{T}}\xspace}
\nc{\bU}{\ensuremath{\mathbf{U}}\xspace}
\nc{\bV}{\ensuremath{\mathbf{V}}\xspace}
\nc{\bW}{\ensuremath{\mathbf{W}}\xspace}
\nc{\bX}{\ensuremath{\mathbf{X}}\xspace}
\nc{\bY}{\ensuremath{\mathbf{Y}}\xspace}
\nc{\bZ}{\ensuremath{\mathbf{Z}}\xspace}

\nc{\dA}{\ensuremath{\mathds{A}}\xspace}
\nc{\dB}{\ensuremath{\mathds{B}}\xspace}
\nc{\dC}{\ensuremath{\mathds{C}}\xspace}
\nc{\dD}{\ensuremath{\mathds{D}}\xspace}
\nc{\dE}{\ensuremath{\mathds{E}}\xspace}
\nc{\dF}{\ensuremath{\mathds{F}}\xspace}
\nc{\dG}{\ensuremath{\mathds{G}}\xspace}
\nc{\dH}{\ensuremath{\mathds{H}}\xspace}
\nc{\dI}{\ensuremath{\mathds{I}}\xspace}
\nc{\dJ}{\ensuremath{\mathds{J}}\xspace}
\nc{\dK}{\ensuremath{\mathds{K}}\xspace}
\nc{\dL}{\ensuremath{\mathds{L}}\xspace}
\nc{\dM}{\ensuremath{\mathds{M}}\xspace}
\nc{\dN}{\ensuremath{\mathds{N}}\xspace}
\nc{\dO}{\ensuremath{\mathds{O}}\xspace}
\nc{\dP}{\ensuremath{\mathds{P}}\xspace}
\nc{\dQ}{\ensuremath{\mathds{Q}}\xspace}
\nc{\dR}{\ensuremath{\mathds{R}}\xspace}
\nc{\dS}{\ensuremath{\mathds{S}}\xspace}
\nc{\dT}{\ensuremath{\mathds{T}}\xspace}
\nc{\dU}{\ensuremath{\mathds{U}}\xspace}
\nc{\dV}{\ensuremath{\mathds{V}}\xspace}
\nc{\dW}{\ensuremath{\mathds{W}}\xspace}
\nc{\dX}{\ensuremath{\mathds{X}}\xspace}
\nc{\dY}{\ensuremath{\mathds{Y}}\xspace}
\nc{\dZ}{\ensuremath{\mathds{Z}}\xspace}

\nc{\bbA}{\ensuremath{\mathbb{A}}\xspace}
\nc{\bbB}{\ensuremath{\mathbb{B}}\xspace}
\nc{\bbC}{\ensuremath{\mathbb{C}}\xspace}
\nc{\bbD}{\ensuremath{\mathbb{D}}\xspace}
\nc{\bbE}{\ensuremath{\mathbb{E}}\xspace}
\nc{\bbF}{\ensuremath{\mathbb{F}}\xspace}
\nc{\bbG}{\ensuremath{\mathbb{G}}\xspace}
\nc{\bbH}{\ensuremath{\mathbb{H}}\xspace}
\nc{\bbI}{\ensuremath{\mathbb{I}}\xspace}
\nc{\bbJ}{\ensuremath{\mathbb{J}}\xspace}
\nc{\bbK}{\ensuremath{\mathbb{K}}\xspace}
\nc{\bbL}{\ensuremath{\mathbb{L}}\xspace}
\nc{\bbM}{\ensuremath{\mathbb{M}}\xspace}
\nc{\bbN}{\ensuremath{\mathbb{N}}\xspace}
\nc{\bbO}{\ensuremath{\mathbb{O}}\xspace}
\nc{\bbP}{\ensuremath{\mathbb{P}}\xspace}
\nc{\bbQ}{\ensuremath{\mathbb{Q}}\xspace}
\nc{\bbR}{\ensuremath{\mathbb{R}}\xspace}
\nc{\bbS}{\ensuremath{\mathbb{S}}\xspace}
\nc{\bbT}{\ensuremath{\mathbb{T}}\xspace}
\nc{\bbU}{\ensuremath{\mathbb{U}}\xspace}
\nc{\bbV}{\ensuremath{\mathbb{V}}\xspace}
\nc{\bbW}{\ensuremath{\mathbb{W}}\xspace}
\nc{\bbX}{\ensuremath{\mathbb{X}}\xspace}
\nc{\bbY}{\ensuremath{\mathbb{Y}}\xspace}
\nc{\bbZ}{\ensuremath{\mathbb{Z}}\xspace}

\nc{\mrm}[1]{\ensuremath{\mathrm{#1}}\xspace}
\nc{\mit}[1]{\ensuremath{\mathit{#1}}\xspace}
\nc{\mbf}[1]{\ensuremath{\mathbf{#1}}\xspace}
\nc{\mcal}[1]{\ensuremath{\mathcal{#1}}\xspace}
\nc{\msc}[1]{\ensuremath{\mathscr{#1}}\xspace}
\nc{\mfr}[1]{\ensuremath{\mathfrak{#1}}\xspace}

\renc{\bar}[1]{\overline{#1}}
\nc{\sub}{\subset}
\nc{\too}{\longrightarrow}
\nc{\hook}{\hookrightarrow}
\nc*{\hooklongrightarrow}{\ensuremath{\lhook\joinrel\relbar\joinrel\rightarrow}}
\nc{\hooklong}{\hooklongrightarrow}
\nc{\twoheadlongrightarrow}{\relbar\joinrel\twoheadrightarrow}
\nc{\shiso}{\approx}
\nc{\isoto}{\xrightarrow{\sim}}
\nc{\isofrom}{\xleftarrow{\sim}}
\renc{\ge}{\geqslant}
\renc{\le}{\leqslant}

\nc{\id}{\mathrm{id}}

\DeclareMathOperator{\Hom}{\on{Hom}}
\nc{\uHom}{\underline{\smash{\Hom}}}
\DeclareMathOperator{\Maps}{\on{Maps}}

\DeclareMathOperator{\End}{\on{End}}

\nc{\uEnd}{\underline{\smash{\End}}}

\renc{\lim}{\varprojlim}
\makeatletter
\newcommand{\colim@}[2]{%
  \vtop{\m@th\ialign{##\cr
    \hfil$#1\operator@font colim$\hfil\cr
    \noalign{\nointerlineskip\kern1.5\ex@}#2\cr
    \noalign{\nointerlineskip\kern-\ex@}\cr}}%
}
\newcommand{\colim}{%
  \mathop{\mathpalette\colim@{\rightarrowfill@\textstyle}}\nmlimits@
}
\makeatother

\nc{\Cofib}{\on{Cofib}}
\nc{\Fib}{\on{Fib}}
\nc{\initial}{\varnothing}
\nc{\op}{\mathrm{op}}

%% file: macros.tex
\nc{\Spc}{\mrm{Spc}}
\nc{\Spt}{\mrm{Spt}}
\nc{\Spec}{\on{Spec}}
\nc{\Stk}{\mrm{Stk}}
\nc{\Sch}{\mrm{Sch}}
\nc{\aff}{\mrm{aff}}
\nc{\A}{\mbf{A}}
\renc{\P}{\mbf{P}}
\nc{\cl}{{\mrm{cl}}}
\nc{\bDelta}{\mathbf{\Delta}}
\nc{\un}{\mathbf{1}}
\nc{\Tot}{\on{Tot}}
\nc{\Cech}{\textnormal{\v{C}}}
\nc{\Mod}{\mrm{Mod}}
\nc{\Qcoh}{\on{Qcoh}}
\nc{\free}{\mrm{free}}
\nc{\ex}{\mrm{ex}}
\nc{\perf}{\mrm{perf}}
\nc{\aperf}{\mrm{aperf}}
\nc{\coh}{\mrm{coh}}
\nc{\Cat}{\mrm{Cat}}
\nc{\unitm}{\mbf{1}}
\nc{\sphere}{\mbf{S}}
\nc{\Z}{\mbf{Z}}
\nc{\Map}{\mrm{Map}}
\nc{\map}{\mrm{map}}
\nc{\PrL}{\mathcal{P}r^\mrm{L}}
\nc{\PrLst}{\mathcal{P}r^\mrm{L}_\mrm{St}}
\nc{\Motnc}{\mathcal{M}_{\mrm{loc}}}
\nc{\Motadd}{\mathcal{M}_{\mrm{add}}}
\nc{\Einfty}{{\sE_\infty}}
\nc{\E}[1]{{\sE_{#1}}}
\nc{\modmod}{/\!\!/}
\nc{\heart}{\heartsuit}
\nc{\proj}{\mrm{proj}}
\nc{\LL}{\on{L}}
\nc{\K}{\on{K}}
\nc{\G}{\on{G}}
\nc{\GL}{\on{GL}}
\nc{\BGL}{\on{BGL}}
\nc{\M}{\on{M}}
\nc{\KH}{\on{KH}}
\nc{\Alg}{\on{Alg}}
\nc{\CAlg}{\on{CAlg}}
\nc{\cn}{\mrm{cn}}
\nc{\hw}{\mrm{Hw}}
\nc{\htt}{\mrm{Ht}}
\nc{\Fun}{\on{Fun}}
\nc{\Funadd}{\on{Fun}_{\mrm{add}}}
\nc{\Funex}{\on{Fun}_{\mrm{ex}}}
\nc{\Ind}{\on{Ind}}
\nc{\Pro}{\on{Pro}}
\nc{\Kar}{\on{Kar}}
\nc{\Obj}{\on{Obj}}

\nc{\scr}{\term{simplicial commutative ring}}
\nc{\scrs}{\term{simplicial commutative rings}}

\nc{\Einfring}{\term{$\Einfty$-ring}}
\nc{\Einfrings}{\term{$\Einfty$-rings}}

\nc{\Ering}{\term{$\sE_1$-ring}}
\nc{\Erings}{\term{$\sE_1$-rings}}

\nc{\inftyCat}{\term{$\infty$-category}}
\nc{\inftyCats}{\term{$\infty$-categories}}

\nc{\inftyTop}{\term{$\infty$-topos}}
\nc{\inftyTops}{\term{$\infty$-toposes}}

\nc{\inftyGrpd}{\term{$\infty$-groupoid}}
\nc{\inftyGrpds}{\term{$\infty$-groupoids}}

%% file: categorified_spectra.bbl
\newcommand{\etalchar}[1]{$^{#1}$}
\providecommand{\bysame}{\leavevmode\hbox to3em{\hrulefill}\thinspace}
\providecommand{\MR}{\relax\ifhmode\unskip\space\fi MR }
\providecommand{\MRhref}[2]{%
  \href{http://www.ams.org/mathscinet-getitem?mr=#1}{#2}
}
\providecommand{\href}[2]{#2}
\begin{thebibliography}{BGH{\etalchar{+}}19}
\providecommand{\url}[1]{\href{#1}{{\def~{\textasciitilde}\tt #1}}}

\bibitem[AGH19]{AGH}
B.~Antieau, D.~Gepner, and J.~Heller, \emph{K-theoretic obstructions to bounded
  t-structures}, Inventiones mathematicae \textbf{216} (2019), no.~1,
  p.~241–300,  \url{http://dx.doi.org/10.1007/S00222-018-00847-0}

\bibitem[Bar15]{Barwick_heart}
C.~Barwick, \emph{On exact $\infty$-categories and the Theorem of the Heart},
  Compositio Mathematica \textbf{151} (2015), no.~11, p.~2160–2186,
  \url{http://dx.doi.org/10.1112/S0010437X15007447}

\bibitem[Bar16]{BarwickKtheory}
C.~Barwick, \emph{On the algebraic {$K$}-theory of higher categories}, J.
  Topol. \textbf{9} (2016), no.~1, pp.~245--347

\bibitem[Bar22]{barkan2022explicit}
S.~Barkan, \emph{Explicit Square Zero Obstruction Theory}, arXiv:2211.07034,
  2022

\bibitem[BGH{\etalchar{+}}19]{Barwick2019}
C.~Barwick, S.~Glasman, M.~Hoyois, D.~Nardin, and J.~Shah, \emph{Categorifying
  Rationalization}, Forum of Mathematics, Sigma \textbf{7} (2019),
  \url{https://doi.org/10.1017/fms.2019.26}

\bibitem[BGT13]{blumberg2013universal}
A.~J. Blumberg, D.~Gepner, and G.~Tabuada, \emph{A universal characterization
  of higher algebraic {$K$}-theory}, Geom. Topol. \textbf{17} (2013), no.~2,
  pp.~733--838,  \url{https://doi.org/10.2140/gt.2013.17.733}

\bibitem[BKRS22]{BKRS}
T.~Bachmann, A.~A. Khan, C.~Ravi, and V.~Sosnilo, \emph{Categorical Milnor
  squares and K-theory of algebraic stacks}, Selecta Mathematica \textbf{28}
  (2022), no.~5,  \url{http://dx.doi.org/10.1007/s00029-022-00796-w}

\bibitem[BL23]{Burklund_2023}
R.~Burklund and I.~Levy, \emph{On the K-theory of regular coconnective rings},
  Selecta Mathematica \textbf{29} (2023), no.~2,
  \url{http://dx.doi.org/10.1007/s00029-023-00833-2}

\bibitem[Bou79]{Bousfield1979}
A.~Bousfield, \emph{The localization of spectra with respect to homology},
  Topology \textbf{18} (1979), no.~4, pp.~257--281,
  \url{https://doi.org/10.1016/0040-9383(79)90018-1}

\bibitem[BZFN10]{BFN}
D.~Ben-Zvi, J.~Francis, and D.~Nadler, \emph{Integral transforms and Drinfeld
  centers in derived algebraic geometry}, J. Amer. Math. Soc. \textbf{23}
  (2010), no.~4, pp.~909--966,
  \url{https://doi.org/10.1090/S0894-0347-10-00669-7}

\bibitem[CK]{cisinski2017a1homotopy}
D.-C. Cisinski and A.~A. Khan, \emph{$A^1$-homotopy invariance in spectral
  algebraic geometry},
  \href{https://arxiv.org/abs/1705.03340}{arXiv:1705.03340}

\bibitem[DGM13]{dgm-local}
B.~I. Dundas, T.~G. Goodwillie, and R.~McCarthy, \emph{The local structure of
  algebraic {K}-theory}, Algebra and Applications, vol.~18, Springer-Verlag
  London, Ltd., London, 2013

\bibitem[DM94]{Dundas1994}
B.~I. Dundas and R.~McCarthy, \emph{Stable K-Theory and Topological Hochschild
  Homology}, The Annals of Mathematics \textbf{140} (1994), no.~3, p.~685,
  \url{https://doi.org/10.2307/2118621}

\bibitem[ES21]{dgm-elden-vova}
E.~Elmanto and V.~Sosnilo, \emph{On nilpotent extensions of
  {$\infty$}-categories and the cyclotomic trace}, International Mathematics
  Research Notices (2021), pp.~16569--–16633, preprint
  \href{http://arxiv.org/abs/https://academic.oup.com/imrn/advance-article-pdf/doi/10.1093/imrn/rnab179/39312665/rnab179.pdf}{{\sf
  https://academic.oup.com/imrn/advance-article-pdf/doi/10.1093/imrn/rnab179/39312665/rnab179.pdf}}

\bibitem[Fre66]{Freyd1966}
P.~Freyd, \emph{Stable Homotopy}, Springer Berlin Heidelberg, 1966,
  \url{http://dx.doi.org/10.1007/978-3-642-99902-4_5}

\bibitem[Goo90]{GoodwillieK}
T.~G. Goodwillie, \emph{Calculus I: The first derivative of pseudoisotopy
  theory}, K-Theory \textbf{4} (1990), no.~1, p.~1–27,
  \url{http://dx.doi.org/10.1007/BF00534191}

\bibitem[{Hoy}17]{HSS}
{Hoyois, Marc and Scherotzke, Sarah and Sibilla, Nicol\`o}, \emph{Higher
  traces, noncommutative motives, and the categorified {C}hern character}, Adv.
  Math. \textbf{309} (2017), pp.~97--154,
  \url{https://doi.org/10.1016/j.aim.2017.01.008}

\bibitem[Kha19]{Khan_2019}
A.~Khan, \emph{The Morel–Voevodsky localization theorem in spectral algebraic
  geometry}, Geometry \& Topology \textbf{23} (2019), no.~7, p.~3647–3685,
  \url{http://dx.doi.org/10.2140/gt.2019.23.3647}

\bibitem[Kho00]{KhovanovJones}
M.~Khovanov, \emph{A categorification of the Jones polynomial}, Duke
  Mathematical Journal \textbf{101} (2000), no.~3,
  \url{http://dx.doi.org/10.1215/S0012-7094-00-10131-7}

\bibitem[Kho16]{Khovanov_categotification}
\bysame, \emph{Linearization and categorification}, Portugaliae Mathematica
  \textbf{73} (2016), no.~4, p.~319–335,
  \url{http://dx.doi.org/10.4171/PM/1989}

\bibitem[Kle22]{Klemenc_2022}
J.~Klemenc, \emph{The stable hull of an exact $\infty$-category}, Homology,
  Homotopy and Applications \textbf{24} (2022), no.~2, p.~195–220,
  \url{http://dx.doi.org/10.4310/HHA.2022.v24.n2.a9}

\bibitem[KW19]{KasprowskiWinges}
D.~Kasprowski and C.~Winges, \emph{Algebraic K‐theory of stable
  $\infty$‐categories via binary complexes}, Journal of Topology \textbf{12}
  (2019), no.~2, p.~442–462,  \url{http://dx.doi.org/10.1112/topo.12093}

\bibitem[Lev23]{levy2023categorifying}
I.~Levy, \emph{Categorifying reduced rings},
  \href{https://arxiv.org/abs/2303.00263}{arXiv:2303.00263}, 2023

\bibitem[LT19]{Land_2019}
M.~Land and G.~Tamme, \emph{On the {$K$}-theory of pullbacks}, Annals of
  Mathematics \textbf{190} (2019), no.~3, pp.~877--930,
  \url{http://dx.doi.org/10.4007/annals.2019.190.3.4}

\bibitem[Lur17]{HA}
J.~Lurie, \emph{Higher Algebra}, September 2017,
  \url{http://www.math.harvard.edu/~lurie/papers/HA.pdf}

\bibitem[Lur18]{SAG}
\bysame, \emph{Spectral Algebraic Geometry}, February 2018,
  \url{http://www.math.harvard.edu/~lurie/papers/SAG-rootfile.pdf}

\bibitem[{Mat}18]{OWRTC}
{Mathematisches Forschungsinstitut Oberwolfach}, \emph{Arbeitsgemeinschaft:
  Topological Cyclic Homology}, Workshop Reports (2018),
  \url{https://publications.mfo.de/handle/mfo/3637}

\bibitem[Mit90]{Mitchell1990}
S.~A. Mitchell, \emph{The Morava K-theory of algebraic K-theory spectra},
  K-Theory \textbf{3} (1990), no.~6, pp.~607--626,
  \url{https://doi.org/10.1007/bf01054453}

\bibitem[Nee98]{Neeman_HEART_IIIA}
A.~Neeman, \emph{$K$-theory for triangulated categories III(A): The theorem of
  the heart}, Asian Journal of Mathematics \textbf{2} (1998), no.~3,
  pp.~495--589

\bibitem[Nee99]{Neeman_heart_IIIB}
A.~Neeman, \emph{$K$-theory for triangulated categories III(B): The theorem of
  the heart}, Asian Journal of Mathematics \textbf{3} (1999), no.~3,
  pp.~557--608,  \url{http://dx.doi.org/10.4310/ajm.1999.v3.n3.a2}

\bibitem[Nee01a]{Neeman_heart_334}
A.~Neeman, \emph{K-Theory for Triangulated Categories $3\frac{3}{4}$: A Direct
  Proof of the Theorem of the Heart}, K-theory \textbf{22} (2001), pp.~1--144,
  \url{https://api.semanticscholar.org/CorpusID:124886299}

\bibitem[Nee01b]{Neeman2001}
A.~Neeman, \emph{Triangulated Categories}, Princeton University Press, December
  2001,  \url{http://dx.doi.org/10.1515/9781400837212}

\bibitem[Nee21]{Neeman2021}
A.~Neeman, \emph{A counterexample to vanishing conjectures for negative
  K-theory}, Inventiones mathematicae \textbf{225} (2021), no.~2, pp.~427--452,
   \url{https://doi.org/10.1007/s00222-021-01034-4}

\bibitem[NS18]{nikolaus-scholze}
T.~Nikolaus and P.~Scholze, \emph{On topological cyclic homology}, Acta Math.
  \textbf{221} (2018), no.~2, pp.~203--409,
  \url{https://doi.org/10.4310/ACTA.2018.v221.n2.a1}

\bibitem[OS04]{OszvathSzabo}
P.~Ozsvath and Z.~Szabo, \emph{Holomorphic disks and knot invariants}, Advances
  in Mathematics \textbf{186} (2004), no.~1, p.~58–116,
  \url{http://dx.doi.org/10.1016/J.AIM.2003.05.001}

\bibitem[PP23]{patchkoria2023adams}
I.~Patchkoria and P.~Pstr\k{a}gowski, \emph{Adams spectral sequences and
  Franke's algebraicity conjecture}, 2023,
  \href{http://arxiv.org/abs/2110.03669}{{\sf arXiv:2110.03669}}

\bibitem[Ras18]{raskin2018dundasgoodwilliemccarthy}
S.~Raskin, \emph{On the Dundas-Goodwillie-McCarthy theorem},
  \href{https://arxiv.org/abs/1807.06709}{arXiv:1807.06709}, 2018

\bibitem[Sch06]{Schlichting2006}
M.~Schlichting, \emph{Negative K-theory of derived categories}, Mathematische
  Zeitschrift \textbf{253} (2006), no.~1, p.~97–134,
  \url{http://dx.doi.org/10.1007/s00209-005-0889-3}

\bibitem[Sos]{vova-invariance}
V.~Sosnilo, \emph{$\mathbb{A}^1$-invariance of localizing invariants},
  \href{https://arxiv.org/abs/2211.05602}{arXiv:2211.05602}

\bibitem[Sos19]{Sosnilo_2019}
V.~Sosnilo, \emph{Theorem of the Heart in Negative K-Theory for Weight
  Structures}, Documenta Mathematica \textbf{24} (2019), p.~2137–2158,
  \url{http://dx.doi.org/10.4171/dm/722}

\bibitem[Sos22]{sosnilo2022weighted}
\bysame, \emph{Weighted Methods in Noncommutative Geometry}, Ph.D. thesis,
  Higher School of Economics, 2022, available at
  \url{https://www.hse.ru/sci/diss/527514830}

\bibitem[Tam18]{tamme-excision}
G.~Tamme, \emph{Excision in algebraic {$K$}-theory revisited}, Compos. Math.
  \textbf{154} (2018), no.~9, pp.~1801--1814,
  \url{https://doi.org/10.1112/s0010437x18007236}

\bibitem[Tho95]{Thomason1995}
R.~Thomason, \emph{Symmetric monoidal categories model all connective
  spectra.}, Theory and Applications of Categories [electronic only] \textbf{1}
  (1995), pp.~78--118,  \url{http://eudml.org/doc/119148}

\bibitem[Wal78]{WaldK}
F.~Waldhausen, \emph{Algebraic K-theory of topological spaces. I}, I. In
  Algebraic and geometric topology (Proc. Sympos. Pure Math., Stanford Univ.,
  Stanford, Calif., 1976), Part, vol.~1, 1978, pp.~35--60,
  \url{https://pub.uni-bielefeld.de/record/1782192}

\end{thebibliography}
